\documentclass[a4paper,reqno,12pt]{amsart}
\usepackage{t1enc}
\usepackage[margin=1in]{geometry}
\usepackage{ifthen}
\usepackage{graphicx}
\usepackage{color}
\usepackage[footnotesize]{caption}
\usepackage{enumerate}

\newcommand{\ignore}[1]{}

\newcommand{\abs}[1]{\left| #1 \right|}
\newcommand{\expr}[1]{\left( #1 \right)}

\newcommand{\scalar}[1]{\left\langle #1 \right\rangle}
\newcommand{\tscalar}[1]{\langle #1 \rangle}
\newcommand{\A}{\mathcal{A}}

\newcommand{\C}{\mathbf{C}}
\newcommand{\domain}{\mathcal{D}}
\newcommand{\R}{\mathbf{R}}

\newcommand{\pr}{\mathbf{P}}
\newcommand{\ex}{\mathbf{E}}

\newcommand{\ind}{\mathbf{1}}
\newcommand{\eps}{\varepsilon}
\newcommand{\ph}{\varphi}

\newcommand{\thet}{\vartheta}

\newcommand{\sub}{\subseteq}

\newcommand{\pvint}{\pv\!\!\int}

\newcommand{\even}{\mathrm{even}}
\newcommand{\odd}{\mathrm{odd}}

\newcommand{\schwartz}{\mathcal{S}}
\newcommand{\fourier}{\mathcal{F}}
\newcommand{\laplace}{\mathcal{L}}
\newcommand{\cop}{{\R \setminus 0}}
\newcommand{\conv}{*}

\theoremstyle{plain}
\newtheorem{theorem}{Theorem}[section]
\newtheorem{lemma}[theorem]{Lemma}
\newtheorem{corollary}[theorem]{Corollary}
\newtheorem{proposition}[theorem]{Proposition}

\theoremstyle{definition}
\newtheorem{definition}[theorem]{Definition}
\newtheorem{example}[theorem]{Example}
\newtheorem{remark}[theorem]{Remark}

\theoremstyle{remark}

\numberwithin{equation}{section}

\DeclareMathOperator{\sign}{sign}
\DeclareMathOperator{\supp}{supp}

\DeclareMathOperator{\re}{Re}
\DeclareMathOperator{\im}{Im}
\DeclareMathOperator{\pv}{pv}
\DeclareMathOperator{\Arg}{Arg}

\newcommand{\formula}[2][nolabel]
{\ifthenelse{\equal{#1}{nolabel}}
 {\begin{align*} #2 \end{align*}}
 {\ifthenelse{\equal{#1}{}}
  {\begin{align} #2 \end{align}}
  {\begin{align} \label{#1} #2 \end{align}}
 }
}

\title[Symmetric L{\'e}vy processes killed upon hitting $0$]{Spectral theory for symmetric one-dimensional L{\'e}vy processes killed upon hitting the origin}
\author{Mateusz Kwa{\'s}nicki}
\thanks{Work supported by the Polish Ministry of Science and Higher Education grant no.\ N~N201 373136}
\thanks{The author received financial support of the Foundation for Polish Science}
\address{Institute of Mathematics \\ Polish Academy of Sciences \\ ul. {\'S}niadeckich 8, 00-976 Warszawa, Poland}
\email{m.kwasnicki@impan.pl}
\address{Institute of Mathematics and Computer Science \\ Wroc{\l}aw University of Technology \\ ul. Wybrze{\.z}e Wyspia{\'n}\-skiego 27, 50-370 Wroc{\l}aw, Poland}
\email{mateusz.kwasnicki@pwr.wroc.pl}

\begin{document}

\sloppy

\begin{abstract}
Spectral theory for transition operators of one-dimensional symmetric L{\'e}vy process killed upon hitting the origin is studied. Under very mild assumptions, an integral-type formula for eigenfunctions is obtained, and eigenfunction expansion of transition operators and the generator is proved. As an application, and the primary motivation, integral fomulae for the transition density and the distribution of the hitting time of the origin are given in terms of the eigenfunctions.
\end{abstract}

\maketitle

%
%                            ---------- o ----------
%

\section{Introduction}

In two recent articles~\cite{bib:k10, bib:kmr11}, spectral theory for some symmetric L{\'e}vy processes killed upon leaving the half-line was developed. One of the main motivations for these research came from fluctuation theory for L\'evy processes: the distribution of the supremum functional and the first passage time can be expressed in terms of the eigenfunctions of the corresponding transition semigroup. The purpose of the present paper is to obtain similar results for processes killed upon hitting the origin, and apply them to the study of the hitting time of a single point. The following theorem is our main result.

\begin{theorem}
\label{th:taux}
Let $X_t$ be a symmetric one-dimensional L{\'e}vy process, starting at $0$, with L{\'e}vy-Khintchine exponent $\Psi(\xi)$, and suppose that $\Psi'(\xi) > 0$ and $2 \xi \Psi''(\xi) \le \Psi'(\xi)$ for $\xi > 0$, and that $1 / (1 + \Psi(\xi))$ is integrable. Let $\tau_x$ be the first hitting time of $\{x\}$. Then
\formula[eq:taux]{
 \pr(t < \tau_x < \infty) & = \frac{1}{\pi} \int_0^\infty \frac{\cos \thet_\lambda e^{-t \Psi(\lambda)} \Psi'(\lambda) F_\lambda(x)}{\Psi(\lambda)} \, d\lambda
}
for $t > 0$ and almost all $x \in \R$. Here $F_\lambda$ is a bounded, continuous function, defined by
\formula[eq:f-def]{
 F_\lambda(x) & = \sin(|\lambda x| + \thet_\lambda) - G_\lambda(x)
}
for $x \in \R$, where
\formula[eq:theta-def]{
 \thet_\lambda & = \arctan\expr{\frac{1}{\pi} \, \int_0^\infty \expr{\frac{2 \lambda}{\lambda^2 - \xi^2} - \frac{\Psi'(\lambda)}{\Psi(\lambda) - \Psi(\xi)}} d\xi} ,
}
and $G_\lambda$ is an $L^2(\R) \cap C_0(\R)$ function with (integrable) Fourier transform
\formula[eq:g-def]{
 \fourier G_\lambda(\xi) & = \cos \thet_\lambda \expr{\frac{2 \lambda}{\lambda^2 - \xi^2} - \frac{\Psi'(\lambda)}{\Psi(\lambda) - \Psi(\xi)}}
}
for $\xi \in \R \setminus \{-\lambda, \lambda\}$.
\end{theorem}

\begin{remark}
It is easy to check that formula~\eqref{eq:taux} can be differentiated in the time variable $t$ under the integral sign. By combining it with estimates of $F_\lambda$, in many cases one can obtain asymptotic estimates of the density function of $\tau_x$, as well as its derivatives, in a similar manner as in~\cite{bib:kmr11} for the first passage time. This topic will be addressed in a separate article.
\end{remark}

\begin{remark}
When the L\'evy measure of $X_t$ has completely monotone density on $(0, \infty)$, then $G_\lambda(x)$ is completely monotone on $[0, \infty)$ and in many cases can be given by a more straightforward formula; see Theorem~\ref{th:glambda}.
\end{remark}

Let us discuss shortly the assumptions of Theorem~\ref{th:taux}. Since the process $X_t$ is assumed to be symmetric, $\Psi(\xi)$ is a real-valued function and $\Psi(\xi) \ge 0$. The assumption $\Psi'(\xi) > 0$ and $2 \xi \Psi''(\xi) \le \Psi'(\xi)$ for $\xi > 0$ is equivalent to the condition $\psi'(\xi) > 0$, $\psi''(\xi) \le 0$ for $\xi > 0$ for the function $\psi(\xi) = \Psi(\sqrt{\xi})$. This is clearly satisfied by all subordinate Brownian motions (and hence for symmetric stable processes and mixtures of such), but also for many less regular processes, such as truncated symmetric stable processes. Examples are discussed in Section~\ref{sec:examples}. Integrability of $1 / (1 + \Psi(\xi))$ asserts that the process $X_t$ hits single points with positive probability.

The class of L{\'e}vy processes studied in~\cite{bib:k10, bib:kmr11} consisted of symmetric processes, whose L{\'e}vy measure admits a completely monotone density function on $(0, \infty)$. This regularity assumption was needed for an application of the Wiener-Hopf method for solving a certain integral equation in half-line. For the case of hitting a single point, considered below, a more direct approach is available, and therefore much more general processes can be dealt with.

From now on we consider the L\'evy process $X_t$ starting at a fixed point $x \in \R$, and denote the corresponding probability and expectation by $\pr_x$ and $\ex_x$. The functions $F_\lambda$ in Theorem~\ref{th:taux} are eigenfunctions of transition operators of $X_t$ killed upon hitting $\{0\}$. These operators are defined by the formula
\formula{
 P^\cop_t f(x) & = \ex_x(f(X_t) \ind_{\{t < \tau_0\}})
}
for $t > 0$, $x \in \cop$, and they act on $L^p(\cop)$ for arbitrary $p \in [1, \infty]$. By $\A_\cop$ and $\domain(\A_\cop; L^p)$ we denote the generator of the transition semigroup $P^\cop_t$ acting on $L^p(\cop)$, and its domain; a more detailed discussion of these notions is given in Preliminaries. Our main results about $F_\lambda$ are contained in the three theorems stated below.

\begin{theorem}
\label{th:flambda}
Let $X_t$ be a symmetric one-dimensional L{\'e}vy process with L{\'e}vy-Khintchine exponent $\Psi(\xi)$, and suppose that $\Psi'(\xi) > 0$ for $\xi > 0$, and that $1 / (1 + \Psi(\xi))$ is integrable. Fix $\lambda > 0$, and let $F_\lambda$ be defined as in Theorem~\ref{th:taux}. Then $F_\lambda(0) = 0$, $F_\lambda \in \domain(\A_\cop; L^\infty)$, and
\formula[eq:f-eigen]{
 \A_{\cop} F_\lambda & = -\Psi(\lambda) F_\lambda , & P^{\cop}_t F_\lambda & = e^{-t \Psi(\lambda)} F_\lambda
}
for $t > 0$. In addition, for $\xi > 0$,
\formula[eq:f-laplace]{
\begin{aligned}
 \int_{-\infty}^\infty F_\lambda(x) e^{-\xi |x|} dx & = 2 \, \frac{\xi \sin \thet_\lambda + \lambda \cos \thet_\lambda}{\xi^2 + \lambda^2} \\
 & \hspace*{-4em} - \frac{2 \cos \thet_\lambda}{\pi} \, \int_0^\infty \frac{\xi}{\xi^2 + \zeta^2} \expr{\frac{2 \lambda}{\lambda^2 + \xi^2} - \frac{\Psi'(\lambda)}{\Psi(\lambda) - \Psi(\zeta)}} d\zeta .
\end{aligned}
}
In the sense of Schwartz distributions,
\formula[eq:f-fourier]{
 \fourier F_\lambda(\xi) & = \cos \thet_\lambda \pv \frac{\Psi'(\lambda)}{\Psi(\lambda) - \Psi(\xi)} + \pi \sin \thet_\lambda (\delta_\lambda(\xi) + \delta_{-\lambda}(\xi))
}
($\pv$ stands for the principal value), and
\formula[eq:f-gen]{
 \A F_\lambda(x) & = -\Psi(\lambda) F_\lambda(x) + \Psi'(\lambda) \cos \thet_\lambda \, \delta_0(x) ,
}
where $\delta_\xi$ is the Dirac delta distribution at $\xi$.
\end{theorem}

\begin{remark}
The assertions of the theorem can be interpreted correctly also when $1 / (1 + \Psi(\xi))$ is not integrable. In this case $\thet_\lambda = \arctan(\infty) = \pi/2$, and therefore $\cos \thet_\lambda = 0$, $G_\lambda = 0$, and finally $F_\lambda(x) = \cos(\lambda x)$, as for the free process. This reflects the fact that $X_t$ does not hit $0$ from almost all starting points, and therefore $P_t^{\cop} f(x) = P_t f(x)$ for almost all $x \in \R$ ($P_t$ is the transition semigroup of the free process $X_t$).
\end{remark}

\begin{remark}
For a given $\lambda > 0$, the statement of the theorem remains true if the condition $\Psi'(\xi) > 0$ for all $\xi > 0$ is replaced by the condition $\Psi'(\lambda) > 0$ and $\Psi(\xi) \ne \Psi(\lambda)$ for all $\xi \ne \pm \lambda$. It is an interesting open problem to study the remaining cases, namely: when $\Psi$ is increasing on $(0, \infty)$ and $\Psi'(\lambda) = 0$, and when $\Psi(\xi) = \Psi(\lambda)$ has more than one solution on $(0, \infty)$.
\end{remark}

For estimates and more detailed properties of the eigenfunctions, further regularity of the L{\'e}vy-Khintchine exponent $\Psi(\xi)$ is needed. The required assumption is the same as in~\cite{bib:k10, bib:kmr11}, and it can be put in the following three equivalent forms. The notion of a complete Bernstein function is discussed in Preliminaries.

\begin{proposition}[Proposition~2.14 in~\cite{bib:k10}]
\label{prop:cbf}
Let $X_t$ be a one-di\-men\-sion\-al L{\'e}vy process. The following conditions are equivalent:
\begin{enumerate}
\item[(a)] $\Psi(\xi) = \psi(\xi^2)$ for a \emph{complete Bernstein function} $\psi(\xi)$;
\item[(b)] $X_t$ is a \emph{subordinate Brownian motion} (that is, $X_t = B_{Z_t}$, where $B_s$ is the Brownian motion, $Z_t$ is a non-decreasing L{\'e}vy process, and $B_s$ and $Z_t$ are independent processes), and the L{\'e}vy measure of the underlying subordinator $Z_t$ has completely monotone density function;
\item[(c)] $X_t$ is symmetric and the L{\'e}vy measure of $X_t$ has a completely monotone density function on $(0, \infty)$.
\end{enumerate}
\end{proposition}

\begin{theorem}
\label{th:glambda}
Let $\lambda > 0$. Suppose that the assumptions of Theorem~\ref{th:flambda} hold true.
\begin{enumerate}
\item[(a)]
If $2 \xi \Psi''(\xi) \le \Psi'(\xi)$ for all $\xi > 0$, then we have $\thet_\lambda \in [0, \pi/2)$ and $|G_\lambda(x)| \le G_\lambda(0) = \sin \thet_\lambda$. 
\item[(b)]
Suppose that any of the equivalent conditions of Proposition~\ref{prop:cbf} holds true. Then $\thet_\lambda \in [0, \pi/2)$, $G_\lambda$ is completely monotone on $[0, \infty)$, and $G_\lambda(x) = \laplace \gamma_\lambda(|x|)$ (where $\laplace$ is the Laplace transform) for a finite measure $\gamma_\lambda$. In particular, $G_\lambda(x) \ge 0$, and $G_\lambda$ is integrable,
\formula[eq:gintegral]{
 \int_{-\infty}^\infty G_\lambda(x) dx & = \fourier G_\lambda(0) = \frac{\cos \thet_\lambda}{\lambda} \expr{1 - \frac{\lambda \Psi'(\lambda)}{2 \Psi(\lambda)}} \, .
}
If in addition $\Psi$ extends to a function $\Psi^+(\xi)$ holomorphic in the right complex half-plane $\{\xi \in \C : \re \xi > 0\}$ and continuous in $\{\xi \in \C : \re \xi \ge 0\}$, and $\Psi^+(i \xi) \ne \Psi(\lambda)$ for all $\xi > 0$, then for all $x \in \R$,
\formula[eq:g]{
 G_\lambda(x) & = \frac{\Psi'(\lambda) \cos \thet_\lambda}{\pi} \int_0^\infty \im \frac{1}{\Psi(\lambda) - \Psi^+(i \xi)} \, e^{-\xi |x|} d\xi .
}
\end{enumerate}
\end{theorem}

Note that the assumptions for part~(a) of the theorem are exactly the same as in Theorem~\ref{th:taux}.

Before the statement of the third main result about the eigenfunctions $F_\lambda$, let us make the following remark. The `sine-cosine' Fourier transform $\tilde{\fourier}$, defined by the formula
\formula{
 \tilde{\fourier} f(\lambda) & = \expr{\int_{-\infty}^\infty f(x) \cos(\lambda x) dx, \int_{-\infty}^\infty f(x) \sin(\lambda x) dx}
}
for $f \in C_c(\R)$ and $\lambda > 0$, extends to a unitary (up to a constant factor $\sqrt{\pi}$) mapping of $L^2(\R)$ onto $L^2((0, \infty)) \oplus L^2((0, \infty))$, in which the free transition operators $P_t$ take a diagonal form. Namely, $P_t$ transforms to a multiplication operator $e^{-t \Psi(\lambda)}$, that is,
\formula{
 \tilde{\fourier} P_t f(\lambda) & = e^{-t \Psi(\lambda)} \tilde{\fourier} f(\lambda) ;
}
see Remark~\ref{rem:free}. The above formula is an explicit spectral decomposition, or \emph{generalised eigenfunction expansion} of $P_t$. The next theorem provides a similar result for transition operators $P^{\cop}_t$ of the killed process. The corresponding transform is given by a similar formula as $\tilde{\fourier}$, with $\cos(\lambda x)$ replaced by $F_\lambda(x)$.

\begin{theorem}
\label{th:spectral}
Suppose that the assumptions of Theorem~\ref{th:flambda} are satisfied, and let $F_\lambda$ be the eigenfunctions from Theorem~\ref{th:flambda}. Define
\formula{
 \Pi_\even f(\lambda) & = \int_{-\infty}^\infty f(x) F_\lambda(x) dx , \\
 \Pi_\odd f(\lambda) & = \int_{-\infty}^\infty f(x) \sin(\lambda x) dx , \\
 \Pi f(\lambda) & = \expr{\Pi_\even f(\lambda), \Pi_\odd f(\lambda)}
}
for $f \in C_c(\cop)$ and $\lambda > 0$. Furthermore, define
\formula{
 \Pi^*(f_1, f_2)(x) & = \int_0^\infty f_1(\lambda) F_\lambda(x) d\lambda + \int_0^\infty f_2(\lambda) \sin(\lambda x) d\lambda
}
for $f_1, f_2 \in C_c((0, \infty))$ and $x > 0$. Then $\pi^{-1/2} \Pi$ and $\pi^{-1/2} \Pi^*$ extend to unitary operators, which map $L^2(\cop)$ onto $L^2((0, \infty)) \oplus L^2((0, \infty))$, and $L^2((0, \infty)) \oplus L^2((0, \infty))$ onto $L^2(\cop)$, respectively. Furthermore, $\Pi \Pi^* = \pi$, $\Pi^* \Pi = \pi$, and
\formula[eq:fullrepresentation]{
 \Pi P^{\cop}_t f(\lambda) = e^{-t \Psi(\lambda)} \Pi f(\lambda) ,
}
for all $f \in L^2(\cop)$. In addition, $f \in L^2(\cop)$ belongs to $\domain(\A_\cop; L^2)$ if and only if $\Psi(\lambda) \Pi f(\lambda)$ is square integrable on $(0, \infty)$, and in this case for $\lambda > 0$,
\formula[eq:fullrepresentation2]{
 \Pi \A_\cop f(\lambda) = -\Psi(\lambda) \Pi f(\lambda) .
}
\end{theorem}

In a rather standard way, explicit spectral representation for $P^\cop_t$ yields a formula for the kernel of the operator, that is, for the transition density.

\begin{corollary}
\label{cor:pdt}
Suppose that the assumptions of Theorem~\ref{th:flambda} are satisfied, and $2 \xi \Psi''(\xi) \le \Psi'(\xi)$ for all $\xi > 0$. Then the \emph{transition density} of the process killed upon hitting the origin (that is, the kernel function of $P^\cop_t$) is given by
\formula[eq:pdt]{
\begin{aligned}
 p^{\cop}_t(x, y) & = \frac{1}{\pi} \int_0^\infty e^{-t \Psi(\lambda)} (F_\lambda(x) F_\lambda(y) + \sin(\lambda x) \sin(\lambda y)) d\lambda \\
 & = \frac{p_t(x - y) - p_t(x + y)}{2} + \frac{1}{\pi} \int_0^\infty e^{-t \Psi(\lambda)} F_\lambda(x) F_\lambda(y) d\lambda .
\end{aligned}
}
\end{corollary}

Formula~\eqref{eq:pdt}, however, is of rather limited applications due to many cancellations in the integral with two oscillatory terms. Nevertheless, it is one of the very few explicit descriptions of the transition density of a killed L\'evy process.

By symmetry, we have $\pr_x(\tau_0 > t) = \pr_0(\tau_x > t)$. Hence, Theorem~\ref{th:taux} follows trivially from the following result. In the theorem, an extra condition on $\Psi$ is required in order to use Theorem~\ref{th:glambda}(a).

\begin{theorem}
\label{th:tau}
Suppose that the assumptions of Theorem~\ref{th:flambda} are satisfied, and $2 \xi \Psi''(\xi) \le \Psi'(\xi)$ for all $\xi > 0$. Then
\formula[eq:tau]{
 \pr_x(t < \tau_0 < \infty) & = \frac{1}{\pi} \int_0^\infty \frac{\cos \thet_\lambda e^{-t \Psi(\lambda)} \Psi'(\lambda) F_\lambda(x)}{\Psi(\lambda)} \, d\lambda
}
for $t > 0$ and almost all $x \in \cop$.
\end{theorem}

Since $\pr_x(\tau_0 > t) = \int_{\cop} p^\cop_t(x, y) dy$, one could naively try to derive formula~\eqref{eq:tau} by integrating~\eqref{eq:pdt} over $y \in \cop$. Heuristically, changing the order of integration and substituting $\fourier F_\lambda(0) = \cos \thet_\lambda \Psi'(\lambda) / \Psi(\lambda)$ for the integral $\int_{-\infty}^\infty F_\lambda(x) dx$ (which is \emph{not} absolutely convergent, or even convergent in the usual way) would yield~\eqref{eq:tau} with $\pr_x(t < \tau_0)$ instead of $\pr_x(t < \tau_0 < \infty)$. Hence, during this illegitimate use of Fubini, $\pr_0(\tau_0 = \infty)$ is lost.

Relatively little is known about the distribution of the hitting time of single points, $\tau_x$. In the general case, the double integral transform (Fourier in space, Laplace in time) is well-known,
\formula{
 \int_\R e^{i \xi x} \ex_x e^{-z \tau_0} dx & = \frac{c(z)}{z + \Psi(\xi)}
}
for $z > 0$ and $\xi \in \R$, where $c(z)$ is the normalisation constant, chosen so that the integral of the right-hand side is $1$ (see Preliminaries). Hence,
\formula{
 \ex_x e^{-z \tau_0} & = \frac{u_z(x)}{u_z(0)}
}
for $z > 0$ and $x \in \R$, where $u_z(x)$ is the resolvent (or $z$-potential) kernel for $X_t$. However, $u_z(x)$ is typically given by an oscillatory integral (namely, the inverse Fourier transform of $1 / (z + \Psi(\xi))$, and therefore it is not easy to invert the Laplace transform in time, even for symmetric $\alpha$-stable processes. Some asymptotic analysis of $\pr_x(\tau_0 > t)$ (and many other related objects) can be found, for example, in~\cite{bib:bgr61, bib:k69, bib:yyy09}. Some more detailed results in this area have been obtained for spectrally negative L{\'e}vy processes, see, for example, \cite{bib:d91} and Section~46 in~\cite{bib:s99}. For totally asymmetric $\alpha$-stable processes, a series representation for $\pr_x(\tau_0 > t)$ was obtained in~\cite{bib:p08}.

The hitting time of a single point is closely related to questions about local time and excursions of a L\'evy process away from the origin, see~\cite{bib:c10, bib:y10}. Similar questions arise when killing a L\'evy process is replaced by penalizing it whenever it hits $0$, cf.~\cite{bib:yyy09:penalising}. In fact, using the same method, one can find generalized eigenfunctions for the transition operators of the penalized semigroup. Finally, spectral theory for the half-line has been succesfully applied in the study of processes killed upon leaving the interval~\cite{bib:kkm12, bib:kkms10, bib:k12} (see also~\cite{bib:bk04, bib:bk09}) and higher-dimensional domains~\cite{bib:fg11}. One can expect similar applications of the spectral theory for $\cop$, developed in the present article.

The remaining part of the article is divided into four sections. In Preliminaries, we describe the background on Schwartz distributions, complete Bernstein and Stieltjes functions, and L{\'e}vy processes and their generators. Section~\ref{sec:flambda} contains the derivation of the formula for the eigenfunctions $F_\lambda$. The generalised eigenfunction expansion and formulae for the transition density and the distribution of the hitting time are proved in Section~\ref{sec:spectral}. Finally, examples are studied in Section~\ref{sec:examples}.

%
%                            ---------- o ----------
%

\section{Preliminaries}

\subsection{Schwartz distributions}

Some background on the theory of Schwartz distributions is required for the proof of Theorem~\ref{th:flambda}. A general reference here is~\cite{bib:v02}; a closely related setting is described with more details in~\cite{bib:k10}.

The class of Schwartz functions in $\R$ is denoted by $\schwartz$, and $\schwartz'$ is the space od tempered distributions in $\R$. If $\ph \in \schwartz$ and $F \in \schwartz'$, we write $\tscalar{F, \ph}$ for the pairing of $F$ and $\ph$. The Fourier transform of $\ph \in \schwartz$ is defined by $\fourier \ph(\xi) = \int_{-\infty}^\infty e^{i \xi x} \ph(x) dx$. For $F \in \schwartz'$, $\fourier F$ is the tempered distribution satisfying $\tscalar{\fourier F, \ph} = \tscalar{F, \fourier \ph}$ for all $\ph \in \schwartz$. This definition extends the usual definition of the Fourier transform of $L^1(\R)$ or $L^2(\R)$ functions and finite signed measures on $\R$.

The convolution of $\ph_1, \ph_2 \in \schwartz$ is defined as usual by $\ph_1 * \ph_2(x) = \int_{-\infty}^\infty \ph_1(y) \ph_2(x - y) dy$. When $\ph \in \schwartz$ and $F \in \schwartz'$, then $F * \ph$ is the infinitely smooth function defined by
\formula{
 F * \ph(x) & = \tscalar{F, \ph_x} , && \text{where} & \ph_x(y) & = \ph(x - y) .
}
The convolution does not extend continuously to entire space $\schwartz'$. We say that $F_1, F_2 \in \schwartz'$ are \emph{$\schwartz'$-convolvable} if for all $\ph_1, \ph_2 \in \schwartz$, the functions $F_1 * \ph_1$ and $F_2 * \ph_2$ are convolvable in the usual sense, i.e. the integral $\int_{-\infty}^\infty (F_1 * \ph_1)(y) (F_2 * \ph_2)(x - y) dy$ exists for all $x$. When this is the case, the \emph{$\schwartz'$-convolution} $F_1 \conv F_2$ is the unique distribution $F$ satisfying $F * \ph_1 * \ph_2 = (F_1 * \ph_1) * (F_2 * \ph_2)$ for $\ph_1, \ph_2 \in \schwartz$. There are other non-equivalent definitions of the convolution of distributions; see~\cite{bib:v02}.

The \emph{support} of a distribution $F$ is the smallest closed set $\supp F$ such that $\tscalar{F, \ph} = 0$ for all $\ph \in \schwartz$ such that $\ph(x) = 0$ for $x \in \supp F$. If any of the tempered distributions $F_1$, $F_2$ has compact support, or if $F_1 * \ph$ is bounded and $F_2 * \ph_2$ is integrable for all $\ph \in \schwartz$, then $F_1$ and $F_2$ are automatically $\schwartz'$-convolvable. Note that $\schwartz'$-convolution is \emph{not} associative in general.

If distributions $F_1$, $F_2$ are $\schwartz'$-convolvable, then $\fourier (F_1 \conv F_2) = \fourier F_1 \cdot \fourier F_2$ (the \emph{exchange formula}), where the multiplication of distributions extends standard multiplication of functions in an appropriate manner. Below only a few special cases are discussed, and for the general notion of multiplication of distributions, we refer the reader to~\cite{bib:v02}.

When $F_1$ is a measure and $F_2$ is a function, then their multiplication $F_1 \cdot F_2$ is the measure $F_2(x) F_1(dx)$, as expected. Next example requires the following definition. By writing $F_1 = \pv (1/(x - x_0))$ we mean that
\formula{
 \scalar{F_1, \ph} & = \pvint_{-\infty}^\infty \frac{\ph(x)}{x - x_0} \, dx = \lim_{\eps \to 0} \int_{\R \setminus [x_0-\eps, x_0+\eps]} \frac{\ph(x)}{x} \, dx ;
}
the \emph{principal value integral} $\pvint$ is defined by the right-hand side. When $F_1 = \pv (1 / (x - x_0))$ and $F_2$ is a function vanishing at $x_0$ and differentiable at $x_0$, then $F_1 \cdot F_2$ is the ordinary function $F_2(x) / (x - x_0)$. The third example is the extension of the previous one. Let $f$ be a function vanishing at a finite number of points $x_1$, $x_2$, ..., $x_n$, continuously differentiable in a neighborhood of each $x_j$ with $f'(x_j) \ne 0$, and such that $1 / f(x)$ is bounded outside any neighborhood of $\{x_1, ..., x_n\}$. In this case, $F_1 = \pv (1/f(x))$ is formally defined by
\formula{
 \scalar{F_1, \ph} & = \pvint_{-\infty}^\infty \frac{\ph(x)}{f(x)} \, dx = \lim_{\eps \to 0} \int_{\R \setminus \bigcup_{j = 1}^n [x_j-\eps, x_j+\eps]} \frac{\ph(x)}{f(x)} \, dx .
}
Clearly, $F_1 = \sum_{j = 1}^n f'(x_j) \pv (1 / (x - x_j)) + F$, where $F$ is a genuine function. Hence, by the previous example, we obtain the following result: when $F_1 = \pv (1/f(x))$ is as above and $F_2$ is a function which vanishes at each $x_j$ and is differentiable at each $x_j$, then $F_1 \cdot F_2$ is the function $F_2(x) / f(x)$.

The Laplace transform of a function $F$ on $[0, \infty)$ is defined by $\laplace F(\xi) = \int_0^\infty F(x) e^{-\xi x} dx$. For a (possibly signed) Radon measure $m$ on $[0, \infty)$, we have $\laplace m(\xi) = \int_0^\infty e^{-\xi x} m(dx)$.

\subsection{Stieltjes functions and complete Bernstein functions}

A general reference here is~\cite{bib:ssv10}. A function $f(x)$ is said to be a \emph{complete Bernstein function} (CBF in short) if
\formula[eq:cbf]{
 f(x) & = c_1 + c_2 x + \frac{1}{\pi} \int_0^\infty \frac{x}{x + s} \, \frac{m(ds)}{s} , && x > 0 ,
}
where $c_1 \ge 0$, $c_2 \ge 0$, and $m$ is a Radon measure on $(0, \infty)$ such that $\int \min(s^{-1}, s^{-2}) m(ds) < \infty$. The right-hand side clearly extends to a holomorphic function of $x \in \C \setminus (-\infty, 0]$, and we often identify $f$ with its holomorphic extension. A function $f(x)$ is a \emph{Stieltjes function} if
\formula[eq:stieltjes]{
 f(x) & = \frac{\tilde{c}_1}{x} + \tilde{c}_2 + \frac{1}{\pi} \int_0^\infty \frac{1}{x + s} \, \tilde{m}(ds) , && x > 0 ,
}
where $\tilde{c}_1 \ge 0$, $\tilde{c}_2 \ge 0$, and $\tilde{m}$ is a Radon measure on $(0, \infty)$ such that $\int \min(1, s^{-1}) m(ds) < \infty$. Below we collect the properties of Stieltjes and complete Bernstein functions used in the article.

\begin{proposition}[see~\cite{bib:ssv10}]
\label{prop:cbf:defs}
\begin{enumerate}
\item[(a)] A function $f$ is a CBF if and only if $f(z) \ge 0$ for $z > 0$ and either $f$ is constant, or $f$ extends to a holomorphic function in $\C \setminus (-\infty, 0]$, which leaves the upper and the lower complex half-planes invariant.
\item[(b)] A function $f$ is a Stieltjes function if and only if $f(z) \ge 0$ for $z > 0$ and either $f$ is constant, or $f$ extends to a holomorphic function in $\C \setminus (-\infty, 0]$, which swaps the upper and the lower complex half-planes.
\item[(c)] If $f$ is positive, then the following four conditions are mutually equivalent: $f(x)$ is a CBF, $x / f(x)$ is a CBF, $1 / f(x)$ is a Stieltjes function, $f(x) / x$ is a Stieltjes function.
\item[(d)] A function $f$ is a Stielties function with representation~\eqref{eq:stieltjes} if and only if $f$ is the Laplace transform of the measure $\tilde{c}_2 \delta_0(ds) + (\tilde{c}_1 + \laplace \tilde{m}(s)) ds$.\qed
\end{enumerate}
\end{proposition}

\begin{proposition}[Proposition~7.14(a) in~\cite{bib:ssv10}]
\label{prop:psilambda}
If $\psi$ is a non-constant CBF, then
\formula[eq:psilambda]{
 \psi_\lambda(\xi) & = \frac{1 - \xi / \lambda^2}{1 - \psi(\xi) / \psi(\lambda^2)}
}
(extended continuously at $\lambda^2$, so that $\psi_\lambda(\lambda^2) = \psi(\lambda^2) / (\lambda^2 \psi'(\lambda^2))$) is again a CBF.\qed
\end{proposition}

\begin{proposition}[Corollary~6.3 in~\cite{bib:ssv10}, Proposition~2.18 in~\cite{bib:k10}]
\label{prop:cbf:repr}
\hspace*{0em}\par
\begin{enumerate}
\item[(a)]
Let $f$ be a CBF with representation~\eqref{eq:cbf}. Then
\formula[eq:cbf:constants]{
 c_1 & = \lim_{z \to 0^+} f(z) , & c_2 & = \lim_{z \to \infty} \frac{f(z)}{z} \, ,
}
and
\formula[eq:cbf:jump]{
 m(ds) & = \lim_{\eps \to 0^+} (\im f(-s + i \eps) ds) ,
}
with the limit understood in the sense of weak convergence of measures.
\item[(b)]
Let $f$ be a Stieltjes function with representation~\eqref{eq:stieltjes}. Then
\formula[eq:stieltjes:constants]{
 \tilde{c}_1 & = \lim_{z \to 0^+} (z f(z)) , & \tilde{c}_2 & = \lim_{z \to \infty} f(z) ,
}
and
\formula[eq:stieltjes:jump]{
 \tilde{m}(ds) + \pi \tilde{c}_2 \delta_0(ds) & = \lim_{\eps \to 0^+} (-\im f(-s + i \eps) ds) ,
}
with the limit understood in the sense of weak convergence of measures.\qed
\end{enumerate}
\end{proposition}

\begin{proposition}[Proposition~2.21 in~\cite{bib:k10}]
\label{prop:cbf:ests}
\begin{enumerate}
\item[(a)] If $f$ is a complete Bernstein function, then $|f(z)| \le C(\eps) f(|z|)$ and $|f(z)| \le C(f, \eps) (1 + |z|)$ for $z \in \C$, $|\Arg z| \le \pi - \eps$, $\eps \in (0, \pi)$;
\item[(b)] If $f$ is a Stieltjes function, then $|f(z)| \le C(\eps) f(|z|)$ and $|f(z)| \le C(f, \eps) (1 + |z|^{-1})$ for $z \in \C$, $|\Arg z| \le \pi - \eps$, $\eps \in (0, \pi)$.\qed
\end{enumerate}
\end{proposition}

\subsection{L{\'e}vy processes}

Below we recall some standard notions related to L{\'e}vy processes and describe the setting for the present article. General references for L{\'e}vy processes are~\cite{bib:b98, bib:s99}; for the properties of semigroups of killed L{\'e}vy processes, see~\cite{bib:k10} and the references therein. 

Throughout this article, $X_t$ is a one-dimensional L{\'e}vy process with L{\'e}vy-Khintchine exponent $\Psi(\xi)$; that is,
\formula[eq:lk]{
 \ex_0 e^{i \xi X_t} & = e^{-t \Psi(\xi)} , && t > 0 , \, \xi \in \R .
}
The probability an expectation corresponding to the process starting at $x \in \R$ are denoted by $\pr_x$ and $\ex_x$, respectively. The following assumptions are in force throughout the article:
\begin{enumerate}[(A)]
\item \label{as:sym} $X_t$ is a symmetric process;
\item \label{as:hit} $X_t$ hits every single point in finite time with positive probability.
\end{enumerate}
The first assumption is equivalent to the condition $\Psi(\xi) = \Psi(-\xi)$ for all $\xi \in \R$, and to the condition and $\Psi(\xi) \ge 0$ for $\xi \in \R$. For symmetric processes, Assumption~(\ref{as:hit}) is equivalent to integrability of $1 / (z + \Psi(\xi))$ over $\xi \in \R$ for some (and hence for all) $z > 0$, see Theorem~2 in~\cite{bib:k69}, Theorem~II.19 in~\cite{bib:b98} or Theorem~43.3 in~\cite{bib:s99}. This equivalent form is more convenient in applications.

Both conditions are natural for the spectral theory. When Assumption~\ref{as:sym} fails, i.e. when $X_t$ is not symmetric, the operators $P^\cop_t$ are not self-adjoint (in fact, they even fail to be normal operators). Assumption~\ref{as:hit} asserts that the operators $P^\cop_t$ are not equal to $P_t$; the spectral theory for the latter is trivial, see Remark~\ref{rem:free}.

The transition operators $P_t$ of the free process $X_t$ are defined by
\formula{
 P_t f(x) & = \ex_x f(X_t) = \int_\R f(y) \pr_x(X_t \in dy) , && t > 0 , \, x \in \R ,
}
whenever the integral is absolutely convergent. The operators $P_t$ (acting on $L^2(\R)$) are Fourier multipliers with Fourier symbol $e^{-t \Psi(\xi)}$, and hence they are convolution operators. By assumption~(\ref{as:hit}) and Fourier inversion formula, the convolution kernel is an integrable function $p_t(x)$, called the transition density.

Let $D \sub \R$ be an open set, and let $\tau_D = \inf \{ t \ge 0 : X_t \notin D \}$ be the first time the process $X_t$ exits $D$. The transition operators of the process $X_t$ killed upon leaving $D$ are given by
\formula{
 P^D_t f(x) & = \ex_x (f(X_t) \ind_{t < \tau_D}) = \int_D f(y) \pr_x(X_t \in dy; t < \tau_D) , && t > 0 , \, x \in D .
}
The kernel function $p^D_t(x, y)$ of $P^D_t$ (the transition density of the killed process) exists, and $0 \le p^D_t(x, y) \le p_t(x - y)$ ($t > 0$, $x, y \in D$).

By $L^p(D)$, $p \in [1, \infty)$, we denote the space of real-valued functions $f$ with finite $p$-norm $\|f\|_{L^p(D)} = (\int_D |f(x)|^p dx)^{1/p}$. The space of essentially bounded functions on $D$ is denoted by $L^\infty(D)$. The space of bounded continuous functions is denoted by $C_b(\R)$, and $C_0(\R)$ is the set of continuous functions vanishing at $\pm \infty$. By $C_0(D)$ we denote the space of $C_0(\R)$ functions vanishing on the complement of $D$.

In this article, we only consider $D = \cop$, so that $\tau_D = \tau_0$ is the hitting time of the origin. In this case clearly $L^p(\cop)$ can be identified with $L^p(\R)$; nevertheless, we use the notation $L^p(\cop)$ to emphasise that the corresponding operators are related to the process killed upon hitting the origin.

The following properties of transition semigroups are described in detail in~\cite{bib:k10}. The operators $P_t$ and $P^\cop_t$ form a contraction semigroup on $L^p(\R)$ and $L^p(\cop)$, respectively, where $p \in [1, \infty]$. This semigroup is strongly continuous when $p < \infty$. Also, $P_t$ is a contraction semigroup on $C_b(\R)$ and $C_0(\R)$, strongly continuous on the latter space. By Theorem~II.19 in~\cite{bib:b98} or Theorem~43.3 in~\cite{bib:s99}, $0$ is regular for itself, and so $P^\cop_t$ is a strongly continuous semigroup of operators on $C_0(\cop)$ (see also~\cite{bib:k10}).

We denote the generators of the semigroups $P_t$ and $P^\cop_t$ (on appropriate function spaces) by $\A$ and $\A_\cop$, respectively. The corresponding domains are denoted by $\domain(\A; \mathcal{X})$ and $\domain(\A_\cop; \mathcal{X})$, respectively, where $\mathcal{X}$ indicates the underlying function space. For example, $\domain(\A_\cop; L^\infty)$ is the domain of the generator of the semigroup $P^\cop_t$ of contractions on $L^\infty(\cop)$, that is, the set of those $f \in L^\infty(\cop)$ for which the limit
\formula[eq:gen]{
 \A_\cop f & = \lim_{t \to 0^+} \frac{P^\cop_t f - f}{t}
}
exists in $L^\infty(\cop)$. Note that the limit~\eqref{eq:gen}, if it exists, does not depend on the choice of the underlying function space. Therefore, using a single symbol $\A$ for operators acting on different domains $\domain(\A; \mathcal{X})$ causes no confusion (again, see~\cite{bib:k10}).

The resolvent kernel, or the $z$-potential kernel, is defined by $u_z(x) = \int_0^\infty p_t(x) e^{-z t} dt$. For each $z > 0$, it is a bounded, continuous function of $x \in \R$, and $\fourier u_z(\xi) = 1 / (z + \Psi(\xi))$. Furthermore, $u_z(x)$ is the convolution kernel of the inverse operator of $z - \A$ (acting on $L^2(\R)$). On the other hand, by Theorem~II.19 in~\cite{bib:b98} or Theorem~43.3 in~\cite{bib:s99}, $u_z(x) / u_z(0)$ is the Laplace transform of the distribution of $\tau_0$,
\formula{
 \ex_x e^{-z \tau_0} & = \frac{u_z(x)}{u_z(0)} \, , && z > 0 , \, x \in \R .
}
This identity will be used in the proof of Theorem~\ref{th:tau}.

Following~\cite{bib:k10}, we introduce the \emph{distributional generator} of $X_t$, denoted by an italic letter $A$: we let $A$ be the tempered distribution satisfying $\A \ph = A * \ph$ for all $\ph \in \schwartz$. By symmetry of $X_t$, equivalently we have $\tscalar{A, \ph} = \A \ph(0)$ for $\ph \in \schwartz$. The distributional generator can be thought of as a pointwise extension of the generator $\A$. The following result from~\cite{bib:k10} describes the connection between $A$ and $\A_\cop$, the generator of the killed semigroup.

\begin{lemma}[see Lemma~2.10 in~\cite{bib:k10}]
\label{lem:domain}
Let $F \in C_0(\cop)$. If the distribution $A \conv F$ is equal to a $C_0(\cop)$ function plus a distribution supported at $\{0\}$, then $F \in \domain(\A_\cop; C_0)$ and $\A_\cop F(x) = A \conv F(x)$ for $x \in \cop$.\qed
\end{lemma}

The following technical result has been proved in~\cite{bib:k10} for general domains. We record its version for $\cop$.

\begin{lemma}[see Lemma~2.11 in~\cite{bib:k10}]
\label{lem:smooth}
Let $F \in C_b^\infty(\R)$, that is, $F$ an all derivatives of $F$ are bounded functions on $\R$. Suppose that for some $r > 0$, $F(x) = 0$ for $x \in [-r, r]$, and $\A F(0) = 0$. Then $F \in \domain(\A_\cop; C_0)$ and $\A_\cop F(x) = \A F(x)$ for all $x \in \cop$.\qed
\end{lemma}

\begin{remark}
\label{rem:free}
The spectral theory for the free process is very simple, thanks to the L{\'e}vy Khintchine formula~\eqref{eq:lk}. Indeed, the function $e^{i \lambda x}$ ($\lambda \in \R$) is the eigenfunction of $P_t$ and $\A$, with eigenvalue $e^{-t \Psi(\lambda)}$ and $-\Psi(\lambda)$, respectively. Clearly, these eigenfunctions yield the generalised eigenfunction expansion of $P_t$ and $\A$ by means of Fourier transform.

Since $X_t$ is symmetric, $\Psi(\lambda) = \Psi(-\lambda)$ is real. It follows that $\sin(\lambda x)$ and $\cos(\lambda x)$ ($\lambda > 0$) are also eigenfunctions of $P_t$ and $\A$. Consequently, $\sin(\lambda x)$ yields generalised eigenfunction expansion for odd $L^2(\R)$ functions (via Fourier sine transform), while $\cos(\lambda x)$ plays a similar role for even $L^2(\R)$ functions (via Fourier cosine transform).\qed
\end{remark}

%
%                            ---------- o ----------
%

\section{Eigenfunctions in $\cop$}
\label{sec:flambda}

In this section the formula for $F_\lambda$ is derived, and some properties of eigenfunctions are studied. More precisely, we prove Theorems~\ref{th:flambda} and~\ref{th:glambda}. The argument follows closely the approach of~\cite{bib:k10}, where the case of the half-line $(0, \infty)$ is studied. Noteworthy, in our case there is no need to employ the Wiener-Hopf method, which makes the proof significantly shorter and simpler.

Mimicking the definition of distributional eigenfunctions in half-line (Definition~4.1 in~\cite{bib:k10}), we introduce the notion of distributional eigenfunctions of $\A_\cop$. Note that the condition $F(0) = 0$ has no meaning for general Schwartz distributions $F$, so in contrast to the definition of~\cite{bib:k10}, \emph{at this stage} we do not assume that $F$ vanishes in the complement of the domain (that is, at the origin).

\begin{definition}
\label{def:weak}
A tempered distribution $F$ is said to be a \emph{distributional eigenfunction} of $\A_\cop$, corresponding to the eigenvalue $-\Psi(\lambda)$, if $F$ is $\schwartz'$-convolvable with $A$, and $A \conv F + \psi(\lambda^2) F$ is supported in $\{0\}$.
\end{definition}

Note that, in particular, $\cos(\lambda x)$ is a distributional eigenfunction of $\A_\cop$. However, it is not the one we are looking for, as it does not vanish at $0$. By copying the proof of Lemma~4.2 in~\cite{bib:k10} nearly verbatim, one obtains the following result. We only sketch the prove and omit the technical details, referring the interested reader to~\cite{bib:k10}.

\begin{lemma}
\label{lem:eigenfunctions}
Let $\lambda > 0$, and suppose that $F$ is a distributional eigenfunction of $\A_\cop$, corresponding to the eigenvalue $-\Psi(\lambda)$. If $F$ is bounded, continuous on $\R$, $F(0) = 0$ and
\formula{
 \lim_{x \to \pm \infty} (F(x) - C \sin(|\lambda x| + \thet)) = 0
}
for some $C, \thet \in \R$, then $F \in \domain(\A_\cop; L^\infty)$ and $\A_\cop F = -\Psi(\lambda^2) F$.
\end{lemma}

\begin{proof}[Sketch of the proof]
It is possible to find an infinitely smooth function $f_1$ such that $f_1(x) = 0$ when $|x| \le 1$, $f_1(x) = C \sin(|\lambda x| + \thet)$ when $|x|$ is large enough, and $\A f_1(0) = 0$. We let $f_2(x) = F(x) - f_1(x)$. By Lemma~\ref{lem:smooth}, $f_1 \in \domain(\A_\cop; C_b)$ and $\A_\cop f_1(x) = \A f_1(x) = A \conv f_1(x)$. With a little effort, one shows that Lemma~\ref{lem:domain} applies to $f_2$. It follows that $f_2 \in \domain(\A_\cop; C_0)$ and $\A_\cop f_2(x) = A \conv f_2(x)$. Hence, $F = f_1 + f_2 \in \domain(\A_\cop; C_b)$ and $\A_\cop F = A * F = -\psi(\lambda^2) F$, as desired
\end{proof}

It is relatively easy to find a formula for distributional eigenfunctions of $\A_\cop$ satisfying the assumptions of Lemma~\ref{lem:eigenfunctions}. First, we give a brief, heuristic derivation of the formula. Suppose that $F_\lambda$ is a bounded, continuous, even function on $\R$ such that $F_\lambda(0) = 0$ and $A \conv F + \psi(\lambda^2) F$ is supported in $\{0\}$. A tempered distribution is supported in $\{0\}$ if and only if its Fourier transform is a polynomial. Hence, $(-\Psi(\xi) + \Psi(\lambda)) \fourier F_\lambda(\xi)$ is a polynomial $Q(\xi)$. It follows that the distribution $\fourier F_\lambda$ is expected to have form
\formula{
 \pv \frac{Q(\xi)}{\Psi(\lambda) - \Psi(\xi)}
}
(the principal value corresponds to singularities at $\pm \lambda$), plus some distribution supported in $\{-\lambda, \lambda\}$ (the zeros of $\Psi(\lambda) - \Psi(\xi)$). The function $F_\lambda$ should be as regular as possible, so we assume that $Q$ is constant (say, $Q(\xi) = a_\lambda \Psi'(\lambda)$), and that the distribution supported in $\{-\lambda, \lambda\}$ is a combination of Dirac measures (say, $\pi b_\lambda (\delta_\lambda + \delta_{-\lambda})$). This suggests the following definition:
\formula{
 \fourier F_\lambda(\xi) & = a_\lambda \pv \frac{\Psi'(\lambda)}{\Psi(\lambda) - \Psi(\xi)} + \pi b_\lambda (\delta_\lambda(\xi) + \delta_{-\lambda}(\xi))
}
for some $a_\lambda, b_\lambda \in \R$. We can normalize this definition by assuming that $a_\lambda^2 + b_\lambda^2 = 1$ and $a_\lambda \ge 0$, so that $a_\lambda = \cos \thet_\lambda$ and $b_\lambda = \sin \thet_\lambda$ for some $\thet_\lambda \in [-\pi/2, \pi/2)$. Furthermore, the condition $F_\lambda(0) = 0$ can be formally rewritten as $\int_{-\infty}^\infty \fourier F_\lambda(\xi) d\xi = 0$, which gives a linear equation in $a_\lambda$ and $b_\lambda$. After soving this equation, we obtain formulae given in Theorem~\ref{th:flambda}.

\begin{remark}
\label{rem:pv}
Before the proof Theorem~\ref{th:flambda}, let us make tho following observation. The primitive function of $2 \lambda / (\lambda^2 - \xi^2)$ is $\log((\lambda + \xi) / (\lambda - \xi))$. This implies that for $\lambda > 0$,
\formula{
 \frac{1}{\pi} \, \pvint_0^\infty \frac{2\lambda}{\lambda^2 - \xi^2} \, d\xi & = 0 .
}
By a similar direct calculation (we omit the details), for $\lambda, \xi > 0$,
\formula{
 \frac{1}{\pi} \, \pvint_0^\infty \frac{\xi}{\xi^2 + \zeta^2} \, \frac{2\lambda}{\lambda^2 - \zeta^2} \, d\zeta & = \frac{\lambda}{\lambda^2 + \xi^2} \, .
}
Hence, formulae of Theorem~\ref{th:flambda} take a simpler form
\formula[eq:theta-def-alt]{
 \thet_\lambda & = \arctan\expr{-\frac{1}{\pi} \, \pvint_0^\infty \frac{\Psi'(\lambda)}{\Psi(\lambda) - \Psi(\xi)} \, d\xi} ,
}
and
\formula[eq:f-laplace-alt]{
\begin{aligned}
 \int_{-\infty}^\infty F_\lambda(x) e^{-\xi |x|} dx & = 2 \, \frac{\xi \sin \thet_\lambda}{\xi^2 + \lambda^2} \\
 & \hspace*{-4em} + \frac{2 \cos \thet_\lambda}{\pi} \, \pvint_0^\infty \frac{\xi}{\xi^2 + \zeta^2} \frac{\Psi'(\lambda)}{\Psi(\lambda) - \Psi(\zeta)} \, d\zeta .
\end{aligned}
}
\end{remark}

\begin{proof}[Proof of Theorem~\ref{th:flambda}]
Since $\Psi(\xi)$ is smooth near $\xi = \lambda$, the integrand in~\eqref{eq:theta-def} is a bounded function, and the integral is finite by the assumption. By a similar argument, the right-hand side of~\eqref{eq:g-def} is a bounded integrable function. Hence, \eqref{eq:g-def} indeed defines an $L^2(\R) \cap C_0(\R)$ function $G_\lambda(x)$, and so $F_\lambda(x)$ is well defined and belongs to $C_b(\R)$.

Denote
\formula[eq:k0]{
 K_\lambda & = -\frac{1}{\pi} \, \pvint_0^\infty \frac{\Psi'(\lambda) d\xi}{\Psi(\lambda) - \Psi(\xi)} \, , && \lambda > 0 ,
}
so that $\thet_\lambda = \arctan(K_\lambda)$ (by~\eqref{eq:theta-def-alt}). Let
\formula[eq:abtheta]{
 a_\lambda & = \cos \thet_\lambda = \frac{1}{\sqrt{1 + K_\lambda^2}} \, , & b_\lambda & = \sin \thet_\lambda = \frac{K_\lambda}{\sqrt{1 + K_\lambda^2}} \, .
}
Note that $\fourier G_\lambda(\xi) = a_\lambda (2 \lambda / (\lambda^2 - \xi^2) - \Psi'(\lambda) / (\Psi(\lambda) - \Psi(\xi)))$. We define
\formula{
 F_\lambda^*(x) = \sin(|\lambda x| + \thet_\lambda) = a_\lambda \sin |\lambda x| + b_\lambda \cos(\lambda x) ,
}
so that $F_\lambda(x) = F_\lambda^*(x) - G_\lambda(x)$.

By Fourier inversion formula,
\formula{
 G_\lambda(0) & = \frac{1}{2 \pi} \int_{-\infty}^\infty \fourier G_\lambda(\xi) d \xi \\
 & = \frac{a_\lambda}{\pi} \expr{\pvint_0^\infty \frac{2 \lambda}{\lambda^2 - \xi^2} \, d\xi - \pvint_0^\infty \frac{\Psi'(\lambda)}{\Psi(\lambda) - \Psi(\xi)} \, d\xi} .
}
The first principal value integral vanishes (cf. Remark~\ref{rem:pv}). It follows that
\formula{
 G_\lambda(0) & = a_\lambda K_\lambda = \cos(\thet_\lambda) \tan(\thet_\lambda) = \sin(\thet_\lambda) = F_\lambda^*(0) ,
}
and therefore $F_\lambda(0) = 0$, as desired.

Note that $2 \pi \delta_\lambda(\xi)$ is the Fourier transform of $e^{-i \lambda x}$, while $\pv 2 i / (\lambda - \xi)$ is the Fourier transform of $-e^{-i \lambda x} \sign x$. Hence, $\pi (\delta_{-\lambda}(\xi) + \delta_\lambda(\xi))$ is the Fourier transform of $\cos(\lambda x)$, and $\pv (2 \lambda / (\lambda^2 - \xi^2))$ is the Fourier transform of $\sin |\lambda x|$. It follows that,
\formula{
 \fourier F_\lambda(\xi) & = \fourier F_\lambda^*(\xi) - \fourier G_\lambda(\xi) \\
 & = \expr{\pi b_\lambda (\delta_{-\lambda}(\xi) + \delta_\lambda(\xi)) + a_\lambda \pv \frac{2 \lambda}{\lambda^2 - \xi^2}} \\
 & \hspace*{10em} - a_\lambda \expr{\frac{2 \lambda}{\lambda^2 - \xi^2} - \frac{\Psi'(\lambda)}{\Psi(\lambda) - \Psi(\xi)}} \\
 & = a_\lambda \pv \frac{\Psi'(\lambda)}{\Psi(\lambda) - \Psi(\xi)} + \pi b_\lambda (\delta_\lambda(\xi) + \delta_{-\lambda}(\xi)) ,
}
and~\eqref{eq:f-fourier} is proved. As any bounded function, $F_\lambda$ is $\schwartz'$-convolvable with $A$, and by the exchange formula we have (see Preliminaries)
\formula{
 \fourier (A * F_\lambda + \Psi(\lambda) F_\lambda)(\xi) & = (-\Psi(\xi) + \Psi(\lambda)) \fourier F_\lambda(\xi) = a_\lambda \Psi'(\lambda) .
}
In particular, $A * F_\lambda + \Psi(\lambda) F_\lambda = a_\lambda \Psi'(\lambda) \delta_0$, which proves~\eqref{eq:f-gen}. Formula~\eqref{eq:f-laplace-alt} follows from~\eqref{eq:f-fourier} by the exchange formula, since the Fourier transform of $e^{-a |x|}$ is $2 a / (a^2 + \xi^2)$. By Remark~\ref{rem:pv}, \eqref{eq:f-laplace} follows. It remains to prove that $F_\lambda \in \domain(\A_\cop; L^\infty)$ and that~\eqref{eq:f-eigen} holds true. 

By~\eqref{eq:f-gen}, $F_\lambda$ is a distributional eigenfunction of $\A_\cop$. Furthermore, $F_\lambda \in C_b(\R)$, $F_\lambda(0) = 0$ and $F_\lambda(x) - F_\lambda^*(x) = G_\lambda(x)$ converges to $0$ as $x \to \pm \infty$. By Lemma~\ref{lem:eigenfunctions}, $F_\lambda \in \domain(\A_\cop; L^\infty)$ and $\A_\cop F_\lambda = -\Psi(\lambda) F_\lambda$. Finally, for $x \in \cop$ and $t > 0$ we have (see, for example, \cite{bib:d65})
\formula{
 P^\cop_t F_\lambda(x) & = F_\lambda(x) + \int_0^t P^\cop_s \A_\cop F_\lambda(x) ds \\
 & = F_\lambda(x) - \Psi(\lambda) \int_0^t P^\cop_s F_\lambda(x) ds .
}
For a fixed $x \in \cop$, this integral equation is easily solved, and we obtain $P^\cop_t F_\lambda(x) = e^{-t \Psi(\lambda)} F_\lambda(x)$, as desired.
\end{proof}

Suppose that $\Psi(\xi) = \psi(\xi^2)$ for some function $\psi(\xi)$ ($\xi > 0$). Note that with the notation of~\eqref{eq:psilambda}, we have
\formula[eq:fourierg]{
\begin{aligned}
 \fourier G_\lambda(\xi) & = a_\lambda \expr{\frac{2 \lambda}{\lambda^2 - \xi^2} - \frac{2 \lambda \psi'(\lambda^2)}{\psi(\lambda^2) - \psi(\xi^2)}} \\
 & = a_\lambda \, \frac{2 \lambda}{\lambda^2 - \xi^2} \expr{1 - \frac{\psi_\lambda(\xi^2)}{\psi_\lambda(\lambda^2)}} = \frac{2 a_\lambda}{\lambda} \, \frac{1}{(\psi_\lambda)_\lambda(\xi^2)} \, .
\end{aligned}
}

\begin{proof}[Proof of Theorem~\ref{th:glambda}]
\emph{Part (a).}
Let $\psi(\xi) = \Psi(\sqrt{\xi})$. Then we have $\psi'(\xi) > 0$ and $\psi''(\xi) \le 0$ for $\xi > 0$; hence, $\psi$ is increasing and concave. First, we prove that $\thet_\lambda \ge 0$. Let $K_\lambda$ be defined by~\eqref{eq:k0}. Observe that since $\psi_\lambda(\xi^2)$ is infinitely smooth near $\xi = \lambda$, we have
\formula{
 K_\lambda & = - \frac{1}{\pi} \, \frac{\lambda^2 \Psi'(\lambda)}{\Psi(\lambda)} \, \pvint_0^\infty \frac{\psi_\lambda(\xi^2)}{\lambda^2 - \xi^2} \, d\xi \\
 & = - \frac{1}{\pi} \, \lim_{q \to 1^-} \frac{\lambda^2 \Psi'(\lambda)}{\Psi(\lambda)} \, \expr{\int_0^{q \lambda} + \int_{\lambda / q}^\infty} \frac{\psi_\lambda(\xi^2)}{\lambda^2 - \xi^2} \, d\xi .
}
By substituting $\xi = \lambda s$ for $\xi < \lambda$ and $\xi = \lambda / s$ for $\xi > \lambda$, and then combining the two integrals, we obtain that
\formula{
 K_\lambda & = \frac{1}{\pi} \, \frac{\lambda \Psi'(\lambda)}{\Psi(\lambda)} \, \int_0^1 \frac{(\psi_\lambda(\lambda^2 / s^2) - \psi_\lambda(\lambda^2 s^2))}{1 - s^2} \, ds .
}
Since $\psi$ is concave, $\psi_\lambda$ is a non-decreasing function (e.g. by a geometric argument: $\lambda^{-2} \psi(\lambda^2) / \psi_\lambda(\xi) = (\psi(\xi^2) - \psi(\lambda^2)) / (\xi^2 - \lambda^2)$ is the difference quotient of $\psi$, and hence a non-increasing function of $\xi$). It follows that $K_\lambda \ge 0$, and therefore $\thet_\lambda = \arctan(K_\lambda) \ge 0$. Monotonicity of $\psi_\lambda$ and formula~\eqref{eq:fourierg} imply also that $\fourier G_\lambda(\xi) \ge 0$ for all $\xi$, which, by Fourier inversion formula, proves that $|G_\lambda(\xi)| \le G_\lambda(0)$.

\emph{Part (b).}
Suppose now that $\psi$ defined as above is a complete Bernstein function. Clearly, $\psi$ is concave, so, by part~(a), $\thet_\lambda \in [0, \pi/2)$. The proof of~\eqref{eq:g} is very similar to the proof of formula~(1.4) in~\cite{bib:k10}.

Let $\tilde{G}_\lambda(x) = \ind_{(0, \infty)}(x) G_\lambda(x)$. We will show that $\laplace \tilde{G}_\lambda$ is a Stieltjes function. For $a > 0$, the Fourier transform of $e^{-a |x|}$ is $2 a / (a^2 + \xi^2)$. Hence, by Plancherel's theorem, for $\xi > 0$ we have
\formula{
 \laplace \tilde{G}_\lambda(\xi) & = \frac{1}{2} \int_{-\infty}^\infty e^{-\xi |x|} G_\lambda(x) dx = \frac{1}{\pi} \int_0^\infty \frac{\xi \fourier G_\lambda(\zeta)}{\xi^2 + \zeta^2} \, d\zeta .
}
Substituting $\zeta = s \sqrt{\xi}$, we obtain
\formula[eq:lapg]{
 \laplace \tilde{G}_\lambda(\xi) & = \frac{1}{\pi} \int_0^\infty \frac{\sqrt{\xi} \fourier G_\lambda(s \sqrt{\xi})}{\xi + s^2} \, ds .
}
By~\eqref{eq:fourierg},
\formula{
 \fourier G_\lambda(\xi) & = \frac{2 \cos \thet_\lambda}{\lambda} \, \frac{1}{(\psi_\lambda)_\lambda(\xi^2)} \, .
}
By Proposition~\ref{prop:psilambda}, $\psi_\lambda$ and $(\psi_\lambda)_\lambda$ are complete Bernstein functions. Hence, by Proposition~\ref{prop:cbf:defs}(c) $1 / (\psi_\lambda)_\lambda$ is a Stieltjes function, bounded on $(0, \infty)$, and so $\fourier G_\lambda$ extends to a holomorphic function in the right complex half-plane $\re \xi > 0$. Abusing the notation, we denote this extension with the same symbol $\fourier G_\lambda$. By Proposition~\ref{prop:cbf:ests}(b), $\fourier G_\lambda(\xi)$ is bounded on every region $-\pi/2 + \eps \le \Arg \xi \le \pi/2 - \eps$. It follows that the right-hand side of~\eqref{eq:lapg} defines a holomorphic function in $\C \setminus (-\infty, 0]$, which is the holomorphic extension of $\laplace \tilde{G}_\lambda$. We denote this extension by $g(\xi)$.

Assume that $\im \xi > 0$ and $\re \xi < 0$. By substituting $s = -\zeta \sqrt{\xi}$ ($\zeta$ is a complex variable here) in the right-hand side of~\eqref{eq:lapg}, we obtain
\formula{
 g(\xi) & = -\frac{1}{\pi} \int_0^{-\xi^{-1/2} \infty} \frac{\fourier G_\lambda(-\xi \zeta)}{1 + \zeta^2} \, d\zeta .
}
The integrand has a simple pole at $\zeta = i$, with the corresponding residue $(-i/2) \fourier G_\lambda(-i \xi)$. By an appropriate contour integration, the residue theorem and passage to a limit (for the details, see Lemma~3.8 in~\cite{bib:k10}),
\formula[eq:gbd]{
 g(\xi) & = -\frac{1}{\pi} \int_0^\infty \frac{\fourier G_\lambda(-\xi \zeta)}{1 + \zeta^2} \, d\zeta + \fourier G_\lambda(-i \xi) .
}
When $\im \xi > 0$, $\re \xi > 0$, we have in a similar manner
\formula[eq:gbd2]{
 g(\xi) & = \frac{1}{\pi} \int_0^\infty \frac{\fourier G_\lambda(\xi \zeta)}{1 + \zeta^2} \, d\zeta .
}
Recall that $\fourier G_\lambda(\zeta) = (2 \lambda^{-1} \cos \thet_\lambda) / (\psi_\lambda)_\lambda(\zeta^2)$, and $1 / (\psi_\lambda)_\lambda$ is a Stieltjes function. Hence, $\im \fourier G_\lambda(\zeta) > 0$ when $\re \zeta > 0$ and $\im \zeta > 0$, and $\im \fourier G_\lambda(\zeta) < 0$ when $\re \zeta > 0$ and $\im \zeta < 0$. By Proposition~\ref{prop:cbf:defs}(b), formulae~\eqref{eq:gbd} and~\eqref{eq:gbd2} imply that $g = \laplace \tilde{G}_\lambda$ is indeed a Stieltjes function. Hence, by Proposition~\ref{prop:cbf:defs}(d), $\tilde{G}_\lambda$ is the Laplace transform of a Radon measure $\gamma_\lambda$ on $[0, \infty)$. Since $\tilde{G}_\lambda \in L^2(\R) \cap L^\infty(\R)$, $\gamma_\lambda$ does not charge $\{0\}$, and it is a finite measure.

The assumption for the last part of theorem can be rephrased as follows: the holomorphic extension of $\psi$ to the upper complex half-plane has continuous boundary limit on $(-\infty, 0)$, which will be denoted by $\psi^+$, and $\psi^+(-\xi^2) \ne \psi(\lambda^2)$ for all $\xi > 0$. In this case, by~\eqref{eq:gbd} and~\eqref{eq:fourierg}, the continuous boundary limit $g^+$ of $g$ on $(-\infty, 0)$ approached from the upper half-plane exists, and it satisfies
\formula{
 g^+(-\xi) & = -\frac{1}{\pi} \int_0^\infty \frac{\fourier G_\lambda(\xi \zeta)}{1 + \zeta^2} \, d\zeta + a_\lambda \expr{\frac{2 \lambda}{\lambda^2 + \xi^2} - \frac{2 \lambda \psi'(\lambda^2)}{\psi(\lambda^2) - \psi^+(-\xi^2)}} .
}
By Proposition~\ref{prop:cbf:repr}(b), it follows that
\formula{
 \gamma_\lambda(d\xi) & = \frac{1}{\pi} \, \im g^+(-\xi) d\xi = \frac{2 \lambda \psi'(\lambda^2) a_\lambda}{\pi} \, \im \frac{1}{\psi(\lambda^2) - \psi^+(-\xi^2)} \, d\xi .
}
Formula~\eqref{eq:g} is proved.
\end{proof}

%
%                            ---------- o ----------
%

\section{Eigenfunction expansion in $\cop$}
\label{sec:spectral}

In this section we prove Theorem~\ref{th:spectral}, which states that the system of generalised eigenfunctions $F_\lambda(x)$ and $\sin(\lambda x)$, $\lambda > 0$, is complete in $L^2(\cop)$. Throughout this section assumptions of Theorem~\ref{th:flambda} are in force, that is, $\Psi(\xi)$ is the L{\'e}vy-Khintchine exponent of a symmetric L{\'e}vy process $X_t$, $1 / (1 + \Psi(\xi))$ is integrable, and $\Psi'(\xi) > 0$ for $\xi > 0$. Our argument is modelled after the proof of Theorem~1.9 in~\cite{bib:kmr11}, providing a similar result for the half-line. Noteworthy, in contrast to the previous section, here the case of $\cop$ appears to require essentially more work than the half-line $(0, \infty)$. 

We begin with some auxiliary definitions and four technical lemmas. For $z \in \C \setminus (-\infty, 0]$, we define
\formula{
 \ph(z) & = \frac{1}{\pi} \int_0^\infty \frac{1}{\Psi(\zeta) + z} \, d\zeta , \\
 \ph(\xi, z) & = \frac{1}{\pi} \int_0^\infty \frac{\xi^2}{\xi^2 + \zeta^2} \, \frac{1}{\Psi(\zeta) + z} \, d\zeta , \\
 \ph(\xi_1, \xi_2, z) & = \frac{1}{\pi} \int_0^\infty \frac{\xi_1^2}{\xi_1^2 + \zeta^2} \, \frac{\xi_2^2}{\xi_2^2 + \zeta^2} \, \frac{1}{\Psi(\zeta) + z} \, d\zeta .
}
The integrals are convergent by integrability of $1 / (1 + \Psi(\xi))$. By a substitution $\Psi(\zeta) = s$, one easily sees that $\ph(z)$, $\ph(\xi, z)$ and $\ph(\xi_1, \xi_2, z)$ are Stieltjes functions of $z$. Representing measures (as in Proposition~\ref{prop:cbf:repr}) for these and related Stieltjes functions play an important role in the sequel.

We remark that the above three functions are related to the resolvent (or $z$-potential) kernel $u_z(x)$ of the operator $\A$, for example, $\ph(z) = u_z(0)$ and $\ph(\xi, z) / \xi = \int_{-\infty}^\infty u_z(x) e^{-\xi |x|} dx$ (see Preliminaries). These connections are only used in the proof of Theorem~\ref{th:tau}.

\begin{lemma}
\label{lem:1}
For any $\xi > 0$, the function $\ph(\xi, z) / \ph(z)$ is a Stieltjes function of $z$.
\end{lemma}

\begin{proof}
Clearly, $\ph(\xi, z) / \ph(z) > 0$ for $z > 0$. By Proposition~\ref{prop:cbf:defs}(b), it suffices to show that $\im (\ph(\xi, z) / \ph(z)) \le 0$ when $\im z > 0$; then automatically $\im (\ph(\xi, z) / \ph(z)) = -\im (\ph(\xi, \bar{z}) / \ph(\bar{z})) \ge 0$ when $\im z < 0$. Let $z = x + i y$ for real $x$, $y$. We have
\formula{
 \re \frac{1}{\Psi(\zeta) + z} & = \frac{\Psi(\zeta) + x}{(\Psi(\zeta) + x)^2 + y^2} \, , & \im \frac{1}{\Psi(\zeta) + z} & = -\frac{y}{(\Psi(\zeta) + x)^2 + y^2} \, .
}
Hence, by expanding the integrals and a short calculation, we obtain
\formula{
 \im \frac{\ph(\xi, z)}{\ph(z)} & = \frac{\im \ph(\xi, z) \re \ph(z) - \re \ph(\xi, z) \im \ph(z)}{|\ph(z)|^2} \\
 & = \frac{1}{\pi^2 |\ph(z)|^2} \int_0^\infty \int_0^\infty \frac{\xi^2}{\xi^2 + \zeta_1^2} \times \hspace*{0em} \\
 & \hspace*{5em} \times \frac{y (\Psi(\zeta_1) - \Psi(\zeta_2))}{((\Psi(\zeta_1) + x)^2 + y^2) ((\Psi(\zeta_2) + x)^2 + y^2)} \, d\zeta_1 d\zeta_2 .
}
Another version of the above formula is obtained by changing the roles of $\zeta_1$ and $\zeta_2$. By adding the sides of the two versions of the formula, we obtain a symmetrised version,
\formula{
 2 \im \frac{\ph(\xi, z)}{\ph(z)} & = \frac{1}{\pi^2 |\ph(z)|^2} \int_0^\infty \int_0^\infty \expr{\frac{\xi^2}{\xi^2 + \zeta_1^2} - \frac{\xi^2}{\xi^2 + \zeta_2^2}} \times \hspace*{0em} \\ & \hspace*{4.5em} \times \frac{y (\Psi(\zeta_1) - \Psi(\zeta_2))}{((\Psi(\zeta_1) + x)^2 + y^2) ((\Psi(\zeta_2) + x)^2 + y^2)} \, d\zeta_1 d\zeta_2 .
}
Note that $\Psi(\zeta)$ is increasing on $(0, \infty)$, while $\xi^2 / (\xi^2 + \zeta^2)$ decreases with $\zeta > 0$. It follows that when $y > 0$, the integrand on the right-hand side is non-positive, and so $\im (\ph(\xi, z) / \ph(z)) \le 0$.
\end{proof}

\begin{lemma}
\label{lem:2}
For any $\xi_1, \xi_2 > 0$, the function
\formula{
 \tilde{\ph}(\xi_1, \xi_2, z) & = \ph(\xi_1, \xi_2, z) - \frac{\ph(\xi_1, z) \ph(\xi_2, z)}{\ph(z)}
}
is a Stieltjes function of $z$.
\end{lemma}

\begin{proof}
The argument is very similar to the proof of Lemma~\ref{lem:1}. For simplicity, we denote $\ph(\xi_1, z) = a_1 + i b_1$, $\ph(\xi_2, z) = a_2 + i b_2$, $\ph(z) = a + i b$ for real $a, a_1, a_2, b, b_1, b_2$. We have
\formula[eq:threephi]{
\begin{aligned}
 \im \frac{\ph(\xi_1, z) \ph(\xi_2, z)}{\ph(z)} \hspace*{-8em} & \hspace*{8em} = \frac{(a_1 b_2 + b_1 a_2) a - (a_1 a_2 - b_1 b_2) b}{a^2 + b^2} \\
 & = \frac{(a^2 + b^2) b_1 b_2 - (a b_1 - a_1 b) (a b_2 - a_2 b) }{b (a^2 + b^2)} \\
 & = \frac{\im \ph(\xi_1, z) \im \ph(\xi_2, z)}{\im \ph(z)} + \frac{\im (\ph(\xi_1, z) / \ph(z)) \im (\ph(\xi_2, z) / \ph(z))}{\im (1 / \ph(z))} \, .
\end{aligned}
}
By Lemma~\ref{lem:1}, $\ph(\xi_1, z) / \ph(z)$ and $\ph(\xi_2, z) / \ph(z)$ are Stieltjes functions, and by Proposition~\ref{prop:cbf:defs}(c), $1 / \ph(z)$ is a complete Bernstein function. Hence, the second summand on the right-hand side of~\eqref{eq:threephi} is non-negative when $\im z > 0$ (Proposition~\ref{prop:cbf:defs}(a,b)). The first one is non-positive, but below we prove that it is bounded below by $\im \ph(\xi_1, \xi_2, z)$.

For $z = x + i y$ with real $x$, $y$, we have as in the proof of Lemma~\ref{lem:1},
\formula{
 \hspace*{3em} & \hspace*{-3em} \im \ph(\xi_1, z) \im \ph(\xi_2, z) - \im \ph(z) \im \ph(\xi_1, \xi_2, z) \\
 & = \frac{1}{\pi^2} \int_0^\infty \int_0^\infty \frac{\xi_1^2}{\xi_1^2 + \zeta_1^2} \expr{\frac{\xi_2^2}{\xi_2^2 + \zeta_2^2} - \frac{\xi_2^2}{\xi_2^2 + \zeta_1^2}} \times \hspace*{0em} \\
 & \hspace*{5em} \times \frac{y^2}{((\Psi(\zeta_1) + x)^2 + y^2) ((\Psi(\zeta_2) + x)^2 + y^2)} \, d\zeta_1 d\zeta_2
}
By a similar symmetrisation procedure as in the proof of Lemma~\ref{lem:1}, we obtain that
\formula{
 \hspace*{3em} & \hspace*{-3em} 2 \im \ph(\xi_1, z) \im \ph(\xi_2, z) - 2 \im \ph(z) \im \ph(\xi_1, \xi_2, z) \\
 & = \frac{1}{\pi^2} \int_0^\infty \int_0^\infty \expr{\frac{\xi_1^2}{\xi_1^2 + \zeta_1^2} - \frac{\xi_1^2}{\xi_1^2 + \zeta_2^2}} \expr{\frac{\xi_2^2}{\xi_2^2 + \zeta_2^2} - \frac{\xi_2^2}{\xi_2^2 + \zeta_1^2}} \times \hspace*{0em} \\
 & \hspace*{5em} \times \frac{y^2}{((\Psi(\zeta_1) + x)^2 + y^2) ((\Psi(\zeta_2) + x)^2 + y^2)} \, d\zeta_1 d\zeta_2 ,
}
and again the integrand on the right-hand side is non-positive. Hence,
\formula{
 \im \ph(\xi_1, z) \im \ph(\xi_2, z) - \im \ph(z) \im \ph(\xi_1, \xi_2, z) & \le 0 .
}
When $y > 0$, we have $\im \ph(z) < 0$, and so
\formula{
 \frac{\im \ph(\xi_1, z) \im \ph(\xi_2, z)}{\im \ph(z)} & \ge \im \ph(\xi_1, \xi_2, z) .
}
By~\eqref{eq:threephi}, it follows that $\im \tilde{\ph}(\xi_1, \xi_2, z) \le 0$ whenever $\im z > 0$. Clearly, $\im \tilde{\ph}(\xi_1, \xi_2, z) = -\im \tilde{\ph}(\xi_1, \xi_2, \bar{z}) \ge 0$ when $\im z < 0$. By Proposition~\ref{prop:cbf:defs}(b), it remains to prove that $\tilde{\ph}(\xi_1, \xi_2, z) \ge 0$ when $z > 0$.

Since $\tilde{\ph}(\xi_1, \xi_2, z)$ is real when $z > 0$ and $\im \tilde{\ph}(\xi_1, \xi_2, z) \le 0$ when $\im z > 0$, we have $(\partial / \partial z) \tilde{\ph}(\xi_1, \xi_2, z) \le 0$ for $z > 0$. It follows that $\tilde{\ph}(\xi_1, \xi_2, z)$ is a non-increasing function of $z > 0$, and hence it suffices to show that $\lim_{z \to \infty} \tilde{\ph}(\xi_1, \xi_2, z) \ge 0$. When $z \to \infty$, by the definition of $\ph(\xi, z)$ and $\ph(\xi_1, \xi_2, z)$ and dominated convergence, the functions $\ph(\xi_1, z)$ and $\ph(\xi_1, \xi_2, z)$ converge to $0$. Also, $\ph(\xi_2, z) / \ph(z) \le 1$ for all $z > 0$. We conclude that $\tilde{\ph}(\xi_1, \xi_2, z)$ converges to $0$ as $z \to \infty$, and the proof is complete.
\end{proof}

For $\lambda, \xi > 0$, we define
\formula{
 K_\lambda(\xi) & = - \frac{1}{\pi} \, \pvint_0^\infty \frac{\xi^2}{\xi^2 + \zeta^2} \, \frac{\Psi'(\lambda)}{\Psi(\lambda) - \Psi(\zeta)} \, d\zeta , & L_\lambda(\xi) & = \frac{\xi^2}{\xi^2 + \lambda^2} \, ,
}
and, as in~\eqref{eq:k0},
\formula{
 K_\lambda & = -\frac{1}{\pi} \, \pvint_0^\infty \frac{\Psi'(\lambda) d\xi}{\Psi(\lambda) - \Psi(\xi)} \, .
}
%Note that $\lim_{\xi \to \infty} K_\lambda(\xi) = K_\lambda$ and $\lim_{\xi \to \infty} L_\lambda(\xi) = 1$. 
We denote the boundary limits of $\ph(z)$, $\ph(\xi, z)$ and $\ph(\xi_1, \xi_2, z)$ along $(-\infty, 0)$ approached from the upper half-plane by $\ph^+(-z)$, $\ph^+(\xi, -z)$ and $\ph^+(\xi_1, \xi_2, -z)$ ($z > 0$), respectively. The existence of these limits is a part of the next result.

\begin{lemma}
\label{lem:jump}
For $\xi, \lambda > 0$, we have
\formula{
 \ph^+(-\Psi(\lambda)) & = \frac{K_\lambda - i}{\Psi'(\lambda)} , & \ph^+(\xi, -\Psi(\lambda)) & = \frac{K_\lambda(\xi) - i L_\lambda(\xi)}{\Psi'(\lambda)} \, .
}
Furthermore (cf. Lemma~\ref{lem:1}),
\formula{
 \im \frac{\ph^+(\xi, -\Psi(\lambda))}{\ph^+(-\Psi(\lambda))} & = - \frac{K_\lambda L_\lambda(\xi) - K_\lambda(\xi)}{1 + K_\lambda^2} \, ,
}
and for $\xi_1, \xi_2, \lambda > 0$ (cf. Lemma~\ref{lem:2}),
\formula{
 & \im \tilde{\ph}^+(\xi_1, \xi_2, -\Psi(\lambda)) \\
 & \hspace*{3em} = - \frac{1}{\Psi'(\lambda)} \, \frac{K_\lambda L_\lambda(\xi_1) - K_\lambda(\xi_1)}{\sqrt{1 + K_\lambda^2}} \, \frac{K_\lambda L_\lambda(\xi_2) - K_\lambda(\xi_2)}{\sqrt{1 + K_\lambda^2}} \, .
}
\end{lemma}

\begin{proof}
For $\lambda > 0$ and any bounded, continuously differentiable function $f$ on $(0, \infty)$, we have (see Section~1.8 in~\cite{bib:v02})
\formula{
 \hspace*{3em} & \hspace*{-3em} \lim_{\eps \to 0^+} \int_0^\infty \frac{1}{\Psi(\zeta) - \Psi(\lambda) + i \eps} \, f(\zeta) d\zeta \\
 & = \lim_{\eps \to 0^+} \int_0^\infty \frac{\Psi(\zeta) - \Psi(\lambda)}{(\Psi(\zeta) - \Psi(\lambda))^2 + \eps^2} \, f(\zeta) d\zeta \\
 & \hspace*{5em} - \lim_{\eps \to 0^+} \int_0^\infty \frac{i \eps}{(\Psi(\zeta) - \Psi(\lambda))^2 + \eps^2} \, f(\zeta) d\zeta \\
 & = - \pvint_0^\infty \frac{1}{\Psi(\lambda) - \Psi(\zeta)} \, f(\zeta) d\zeta - \frac{i}{\Psi'(\lambda)} \, f(\lambda) .
}
Hence, the boundary limits $\ph^+(-\Psi(\lambda))$, $\ph^+(\xi, -\Psi(\lambda))$ and $\ph^+(\xi_1, \xi_2, -\Psi(\lambda))$ (approached from the upper half-plane; here and below $\xi, \xi_1, \xi_2 > 0$) are continuous functions of $\lambda > 0$, given by
\formula{
 \ph^+(-\Psi(\lambda)) & = - \frac{1}{\pi} \, \pvint_0^\infty \frac{1}{\Psi(\lambda) - \Psi(\zeta)} \, d\zeta - \frac{i}{\Psi'(\lambda)} = \frac{K_\lambda - i}{\Psi'(\lambda)} \, , \\
 \ph^+(\xi, -\Psi(\lambda)) & = - \frac{1}{\pi} \, \pvint_0^\infty \frac{\xi^2}{\xi^2 + \zeta^2} \, \frac{1}{\Psi(\lambda) - \Psi(\zeta)} \, d\zeta - \frac{i}{\Psi'(\lambda)} \, \frac{\xi^2}{\xi^2 + \lambda^2} \\
 & = \frac{K_\lambda(\xi) - i L_\lambda(\xi)}{\Psi'(\lambda)} \, ,
}
and (we only need the imaginary part)
\formula{
 \im \ph^+(\xi_1, \xi_2, -\Psi(\lambda)) & = -\frac{1}{\Psi'(\lambda)} \, \frac{\xi_1^2}{\xi_1^2 + \lambda^2} \, \frac{\xi_2^2}{\xi_2^2 + \lambda^2} = -\frac{L_\lambda(\xi_1) L_\lambda(\xi_2)}{\Psi'(\lambda)} \, .
}
We obtain that
\formula{
 \im \frac{\ph^+(\xi, -\Psi(\lambda))}{\ph^+(-\Psi(\lambda))} & = - \frac{K_\lambda L_\lambda(\xi) - K_\lambda(\xi)}{1 + K_\lambda^2} \, ,
}
and therefore, by~\eqref{eq:threephi},
\formula{
 \hspace*{0.5em} & \hspace*{-0.5em} \im \frac{\ph^+(\xi_1, -\Psi(\lambda)) \ph^+(\xi_2, -\Psi(\lambda))}{\ph^+(-\Psi(\lambda))} \\
 & = - \frac{L_\lambda(\xi_1) L_\lambda(\xi_2)}{\Psi'(\lambda)} + \frac{K_\lambda L_\lambda(\xi_1) - K_\lambda(\xi_1)}{1 + K_\lambda^2} \, \frac{K_\lambda L_\lambda(\xi_2) - K_\lambda(\xi_2)}{1 + K_\lambda^2} \, \frac{1 + K_\lambda^2}{\Psi'(\lambda)} \\
 & = \frac{1}{\Psi'(\lambda)} \expr{\frac{K_\lambda L_\lambda(\xi_1) - K_\lambda(\xi_1)}{\sqrt{1 + K_\lambda^2}} \, \frac{K_\lambda L_\lambda(\xi_2) - K_\lambda(\xi_2)}{\sqrt{1 + K_\lambda^2}} - L_\lambda(\xi_1) L_\lambda(\xi_2)} .
}
The proof is complete.
\end{proof}

\begin{lemma}
\label{lem:key}
For all $\xi_1, \xi_2 > 0$ and $t \ge 0$, let
\formula[eq:keydef]{
 \Phi(\xi_1, \xi_2, t) & = \int_0^\infty \frac{K_\lambda L_\lambda(\xi_1) - K_\lambda(\xi_1)}{\sqrt{1 + K_\lambda^2}} \, \frac{K_\lambda L_\lambda(\xi_2) - K_\lambda(\xi_2)}{\sqrt{1 + K_\lambda^2}} \, e^{-t \Psi(\lambda)} d\lambda .
}
Then
\formula[eq:key0]{
 \Phi(\xi_1, \xi_2, 0) & = \frac{\pi}{2} \, \frac{\xi_1 \xi_2}{\xi_1 + \xi_2} \, ,
}
and for all $z > 0$,
\formula[eq:key]{
 \frac{1}{\pi} \int_0^\infty \Phi(\xi_1, \xi_2, t) e^{-z t} dt & = \tilde{\ph}(\xi_1, \xi_2, z) .
}
\end{lemma}

\begin{proof}
Note that by Lemma~\ref{lem:jump}, we have
\formula{
 \Phi(\xi_1, \xi_2, t) & = -\int_0^\infty \im \tilde{\ph}^+(\xi_1, \xi_2, -\Psi(\lambda)) e^{-t \Psi(\lambda)} \Psi'(\lambda) d\lambda \\
 & = -\int_0^\infty \im \tilde{\ph}^+(\xi_1, \xi_2, -s) e^{-t s} ds .
}
By Lemma~\ref{lem:2}, $\im \tilde{\ph}^+(\xi_1, \xi_2, -s) \le 0$. Hence, by Fubini,
\formula{
 \int_0^\infty \Phi(\xi_1, \xi_2, t) e^{-z t} dt & = -\int_0^\infty \frac{1}{s + z} \, \im \tilde{\ph}^+(\xi_1, \xi_2, -s) ds .
}
By the definition, the coefficients $c_1, c_2$ in the representation~\eqref{eq:stieltjes} of the Stieltjes function $\ph(\xi_1, \xi_2, z)$ vanish. Hence, by Proposition~\ref{prop:cbf:repr}(b) and the inequality $\tilde{\ph}(\xi_1, \xi_2, z) \le \ph(\xi_1, \xi_2, z)$ ($z > 0$), the coefficients $c_1, c_2$ in the representation~\eqref{eq:stieltjes} for $\tilde{\ph}(\xi_1, \xi_2, z)$ vanish too. By~\eqref{eq:stieltjes} and another application of Proposition~\ref{prop:cbf:repr}(b),
\formula{
 -\int_0^\infty \frac{1}{s + z} \, \im \tilde{\ph}^+(\xi_1, \xi_2, -s) ds & = \pi \tilde{\ph}(\xi_1, \xi_2, z) .
}
This proves~\eqref{eq:key}.

By monotone convergence, $\Phi(\xi_1, \xi_2, t)$ is right-continuous at $t = 0$. This and~\eqref{eq:key} imply that
\formula{
 \Phi(\xi_1, \xi_2, 0) = \lim_{z \to \infty} z \int_0^\infty \Phi(\xi_1, \xi_2, t) e^{-z t} dt = \pi \lim_{z \to \infty} (z \tilde{\ph}(\xi_1, \xi_2, z)) .
}
As $z \to \infty$, by monotone convergence, the function $z \ph(\xi, z)$ converges to $\xi$, but $z \ph(z)$ diverges to infinity. Hence, $z \ph(\xi_1, z) \ph(\xi_2, z) / \ph(z)$ converges to $0$. It follows that (see the definition of $\tilde{\ph}(\xi_1, \xi_2, z)$ in Lemma~\ref{lem:2})
\formula{
 \Phi(\xi_1, \xi_2, 0) = \pi \lim_{z \to \infty} (z \ph(\xi_1, \xi_2, z)) = \int_0^\infty \frac{\xi_1^2}{\xi_1^2 + \zeta^2} \, \frac{\xi_2^2}{\xi_2^2 + \zeta^2} \, d\zeta ;
}
monotone convergence was used again for the second equality. The integral can be easily evaluated (we omit the details), and~\eqref{eq:key0} follows.
\end{proof}

With the above technical background, we can prove generalized eigenfunction expansion of the operators $P^{\cop}_t$. Let $\Pi_\even$, $\Pi_\odd$ and $\Pi$ be defined as in Theorem~\ref{th:spectral}. First, we prove that the operator $\Pi_\even$ is isometric, up to a constant factor.

\begin{lemma}
\label{lem:unitary}
For any even functions $f, g \in L^2(\cop)$,
\formula{
 \scalar{\Pi_\even f, \Pi_\even g}_{L^2((0, \infty))} & = \pi \scalar{f, g}_{L^2(\cop)} .
}
\end{lemma}

\begin{proof}
Let $e_\xi(x) = e^{-\xi |x|}$, $x \in \cop$, $\xi > 0$. By Theorem~\ref{th:flambda}, for $\xi, \lambda > 0$ we have
\formula[eq:piexp]{
\begin{aligned}
 \Pi_\even e_\xi(\lambda) & = \int_{-\infty}^\infty F_\lambda(x) e^{-\xi |x|} dx \\
 & = \frac{-2 K_\lambda(\xi) \cos \thet_\lambda + 2 L_\lambda(\xi) \sin \thet_\lambda}{\xi} \\
 & = \frac{2}{\xi} \, \frac{K_\lambda L_\lambda(\xi) - K_\lambda(\xi)}{\sqrt{1 + K_\lambda^2}} \, .
\end{aligned}
}
By~\eqref{eq:keydef} and~\eqref{eq:key0}, it follows that for $\xi_1, \xi_2 > 0$,
\formula{
 \int_0^\infty \Pi_\even e_{\xi_1}(\lambda) \Pi_\even e_{\xi_2}(\lambda) d\lambda & = \frac{4 \Phi(\xi_1, \xi_2, 0)}{\xi_1 \xi_2} \\
 & = \frac{2 \pi}{\xi_1 + \xi_2} = \pi \scalar{e_{\xi_1}, e_{\xi_2}}_{L^2(\cop)} .
}
The family of functions $e_\xi$ is linearly dense in the space of even $L^2(\cop)$ functions. The result follows by approximation.
\end{proof}

The following side-result is interesting. It is obvious when $\Psi(\xi) = \psi(\xi^2)$ for a complete Bernstein function $\psi$. In the general case, however, direct approach seems problematic.

\begin{proposition}
\label{prop:positivelaplace}
We have
\formula{
 \int_{-\infty}^\infty F_\lambda(x) e^{-\xi |x|} dx & \ge 0 , && \lambda, \xi > 0 .
}
\end{proposition}

\begin{proof}
By Lemma~\ref{lem:jump} and~\eqref{eq:piexp}, we see that for $e_\xi(x) = e^{-\xi |x|}$,
\formula{
 \int_{-\infty}^\infty F_\lambda(x) e^{-\xi |x|} dx & = \Pi_\even e_\xi(\lambda) = \frac{2}{\xi} \, \frac{K_\lambda L_\lambda(\xi) - K_\lambda(\xi)}{\sqrt{1 + K_\lambda^2}} \\
 & = -\frac{2 \sqrt{1 + K_\lambda^2}}{\xi} \, \im \frac{\ph^+(\xi, -\Psi(\lambda))}{\ph^+(-\Psi(\lambda))} \, ,
}
which is nonnegative by Lemma~\ref{lem:1}.
\end{proof}

The space $L^2(\cop)$ is the direct sum of the spaces $L^2_\even(\cop)$ and $L^2_\odd(\cop)$, consisting of even and odd $L^2(\cop)$ functions, respectively. Recall that $\Pi_\odd$ is the `sine part' of the Fourier transform. Hence, $\pi^{-1/2} \Pi_\odd$ is a unitary mapping of $L^2_\odd(\cop)$ onto $L^2((0, \infty))$, and $\Pi_\odd$ is a zero operator on $L^2_\even(\cop)$. In a similar manner, $\pi^{-1/2} \Pi_\even$ is an isometric mapping of $L^2_\even(\cop)$ \emph{into} $L^2((0, \infty))$, and $\Pi_\even$ vanishes on $L^2_\odd(\cop)$. It remains to prove that $\pi^{-1/2} \Pi_\even$ is a unitary operator on $L^2_\even(\cop)$ \emph{onto} $L^2((0, \infty))$. An equivalent condition is that the kernel of the adjoint operator $\Pi_\even^* : L^2((0, \infty)) \to L^2(\cop)$ is trivial. This is proved in the following result.

\begin{lemma}
\label{lem:representation}
For all $f, g \in L^2((0, \infty))$,
\formula[eq:starisometry]{
 \scalar{\Pi_\even^* f, \Pi_\even^* g}_{L^2(\cop)} & = \pi \scalar{f, g}_{L^2((0, \infty))} ,
}
and
\formula[eq:representation]{
 P^{\cop}_t \Pi_\even^* f & = \Pi_\even^* (e^{-t \Psi} f) , && t > 0 .
}
\end{lemma}

\begin{proof}
As in the proof of Lemma~5.2 in~\cite{bib:k10}, formula~\eqref{eq:representation} for $f \in C_c((0, \infty))$ follows directly from Theorem~\ref{th:flambda} and Fubini. Indeed, let $\supp f \sub [\lambda_1, \lambda_2]$, $0 < \lambda_1 < \lambda_2$. Using~\eqref{eq:f-def}, \eqref{eq:g-def} and Fourier inversion formula, it is easy to prove that $F_\lambda(y)$ is bounded uniformly in $\lambda \in [\lambda_1, \lambda_2]$ and $y \in \cop$ (we omit the details). Furthermore, $p^\cop_t(x, y)$ is integrable in $y \in \cop$. Hence,
\formula{
 P^{\cop}_t \Pi_\even^* f & = \int_{\cop} \int_{\lambda_1}^{\lambda_2} p^\cop_t(x, y) F_\lambda(y) f(\lambda) d\lambda dx \\
 & = \int_{\lambda_1}^{\lambda_2} e^{-t \Psi(\lambda)} F_\lambda(x) f(\lambda) d\lambda = \Pi_\even^* (e^{-t \Psi} f) .
}
Since $\Pi_\even^*$ is a bounded operator, \eqref{eq:representation} extends to general $f \in L^2((0, \infty))$ by an approximation argument.

The adjoint of an isometric operator is isometric if and only if the kernel of the adjoint is trivial. Hence, by Lemma~\ref{lem:unitary}, it suffices to show that the kernel of $\Pi_\even^*$ is trivial.

Suppose that $f \in L^2((0, \infty))$ and $\Pi_\even^* f = 0$. We claim that this implies that $f = 0$. By~\eqref{eq:representation}, $\Pi_\even^* (e^{-t \Psi} f) = 0$ for any $t > 0$. Let $e_\xi(x) = e^{-\xi |x|}$, $x \in \cop$. We obtain that for all $t > 0$ and $\xi > 0$,
\formula{
 0 = \scalar{\Pi_\even^* (e^{-t \Psi} f), e_\xi}_{L^2(\cop)} & = \scalar{e^{-t \Psi} f, \Pi_\even e_\xi}_{L^2((0, \infty))} \\
 & = \int_0^\infty e^{-t \Psi(\lambda)} f(\lambda) \Pi_\even e_\xi(\lambda) d\lambda .
}
By substituting $s = \Psi(\lambda)$, we obtain that
\formula{
 \int_0^\infty e^{-t s} \, \frac{f(\Psi^{-1}(s)) \Pi_\even e_\xi(\Psi^{-1}(s))}{\Psi'(\Psi^{-1}(s))} \, ds & = 0 , && t, \xi > 0 .
}
It follows that for each $\xi > 0$, $f(\Psi^{-1}(s)) \Pi_\even e_\xi(\Psi^{-1}(s)) = 0$ for almost all $s > 0$.

In other words, for each $\xi > 0$, $f(\lambda) \Pi_\even e_\xi(\lambda) = 0$ for almost all $\lambda > 0$. By Fubini, for almost every $\lambda > 0$, $f(\lambda) \Pi_\even e_\xi(\lambda) = 0$ for almost all $\xi > 0$. But for each fixed $\lambda > 0$, $\Pi_\even e_\xi(\lambda)$ (as a function of $\xi > 0$) is the Laplace transform of a nonzero function (namely, $2 F_\lambda(x) \ind_{(0, \infty)}(x)$), and so it cannot vanish for almost all $\xi > 0$. This implies that $f(\lambda) = 0$ for almost every $\lambda > 0$, proving our claim.
\end{proof}

\begin{proof}[Proof of Theorem~\ref{th:spectral}]
The operator $\Pi f = (\Pi_\even f, \Pi_\odd f)$ is unitary by Lemmas~\ref{lem:unitary} and~\ref{lem:representation}, and the discussion preceding the statement of Lemma~\ref{lem:representation}.

Let $f \in L^2(\cop)$, and let $f = f_\even + f_\odd$ be the decomposition into the even and odd part. Clearly, $P^\cop_t f_\even$ is an even function, and $P^\cop_t f_\odd$ is an odd function. Hence, $\Pi_\even f_\odd = \Pi_\even P^\cop_t f_\odd = 0$ and $\Pi_\odd f_\even = \Pi_\odd P^\cop_t f_\even = 0$. By Lemma~\ref{lem:representation},
\formula[eq:piplus]{
\begin{aligned}
 \Pi_\even P^{\cop}_t f_\even & = \pi^{-1} \Pi_\even (P^{\cop}_t \Pi_\even^* \Pi_\even f_\even) \\
 & = \pi^{-1} \Pi_\even (\Pi_\even^* (e^{-t \Psi} \Pi_\even f_\even)) = e^{-t \Psi} \Pi_\even f .
\end{aligned}
}
Furthermore, by the strong Markov property, $p_t(x - y) - p^\cop_t(x, y) = \ex_x(p_t(y); \tau_0 \le t)$ (the \emph{Hunt's formula}), and the right-hand side is an even function of $y$ (for any $x \in \cop$). Therefore $P^{\cop}_t f_\odd = P_t f_\odd$. By the remark preceding the statement of Theorem~\ref{th:spectral} (see also Remark~\ref{rem:free}), it follows that
\formula[eq:piminus]{
 \Pi_\odd P^{\cop}_t f_\odd & = \Pi_\odd P_t f_\odd = e^{-t \Psi} \Pi_\odd f_\odd .
}
The above properties combined together prove~\eqref{eq:fullrepresentation}.
\end{proof}

\begin{proof}[Proof of Corollary~\ref{cor:pdt}]
Let $f \in C_c(\cop)$. Then $\Pi f \in L^2((0, \infty)) \oplus L^2((0, \infty))$, and hence $e^{-t \Psi} \Pi f$ is a pair of integrable functions. Hence,
\formula{
 P^{\cop}_t f(x) \hspace*{-3em} & \hspace*{3em} = \pi^{-1} \Pi (e^{-t \Psi} \Pi^* f)(x) \\
 & = \frac{1}{\pi} \biggl( \int_0^\infty e^{-t \Psi(\lambda)} F_\lambda(x) \expr{\int_{-\infty}^\infty f(y) F_\lambda(y) dy} d\lambda \\
 & \hspace*{7em} + \frac{1}{\pi} \int_0^\infty e^{-t \Psi(\lambda)} \sin(\lambda x) \expr{\int_{-\infty}^\infty f(y) \sin(\lambda y) dy} d\lambda \biggr) \\
}
By Theorem~\ref{th:glambda}(a), $F_\lambda(y)$ is bounded uniformly in $\lambda > 0$ and $y \in \cop$. By Fubini, it follows that
\formula{
 P^{\cop}_t f(x) = \frac{1}{\pi} \int_{-\infty}^\infty f(y) \int_0^\infty e^{-t \Psi(\lambda)} (F_\lambda(x) F_\lambda(y) + \sin(\lambda x) \sin(\lambda y)) d\lambda dy .
}
This proves the first equality in~\eqref{eq:pdt}. For the second one, note that
\formula{
 \hspace*{3em} & \hspace*{-3em} \frac{2}{\pi} \int_0^\infty e^{-t \Psi(\lambda)} \sin(\lambda x) \sin(\lambda y) d\lambda \\
 & = \frac{1}{2 \pi} \int_{-\infty}^\infty e^{-t \Psi(\lambda)} (\cos(\lambda (x - y)) - \cos(\lambda (x + y))) d\lambda \\
 & = \frac{1}{2 \pi} \int_{-\infty}^\infty e^{-t \Psi(\lambda)} (e^{-i \lambda (x - y)} - e^{-i \lambda (x + y)}) d\lambda \\
 & = p_t(x - y) - p_t(x + y) . \qedhere
}
\end{proof}

\begin{proof}[Proof of Theorem~\ref{th:tau}]
Recall that we need to prove that
\formula[eq:tau2]{
 \pr_x(t < \tau_0 < \infty) & = \frac{1}{\pi} \int_0^\infty \frac{\cos \thet_\lambda e^{-t \Psi(\lambda)} \Psi'(\lambda) F_\lambda(x)}{\Psi(\lambda)} \, d\lambda .
}
Note that $G_\lambda(0) = \sin \thet_\lambda - F_\lambda(0) = \sin \thet_\lambda$. For $\lambda, t > 0$ and $x \in \cop$ we decompose the integrand in~\eqref{eq:tau2},
\formula[eq:tau-decomposition]{
\begin{aligned}
 \frac{\cos \thet_\lambda e^{-t \Psi(\lambda)} \Psi'(\lambda) F_\lambda(x)}{\Psi(\lambda)} & = \frac{e^{-t \Psi(\lambda)} \Psi'(\lambda)}{\Psi(\lambda)} \, (\sin(|\lambda x| + \thet_\lambda) - \sin \thet_\lambda) \\
 & \hspace*{3.5em} + \frac{e^{-t \Psi(\lambda)} \Psi'(\lambda)}{\Psi(\lambda)} \, (G_\lambda(0) - G_\lambda(x)) .
\end{aligned}
}
By the assumption, $\Psi(\xi) = \psi(\xi^2)$, where $\psi$ is non-negative, increasing and concave. Hence, $\Psi'(\lambda) = 2 \lambda \psi'(\lambda^2) \le 2 \lambda (\psi(\lambda^2) / \lambda^2) = 2 \Psi(\lambda) / \lambda$. It follows that
\formula{
 \abs{\frac{e^{-t \Psi(\lambda)} \Psi'(\lambda)}{\Psi(\lambda)} \, (\sin(|\lambda x| + \thet_\lambda) - \sin \thet_\lambda)} & \le 2 |x| e^{-t \Psi(\lambda)} .
}
Furthermore, the second summand on the right-hand side of~\eqref{eq:tau-decomposition} is non-negative by Theorem~\ref{th:glambda}. It follows that the integral in~\eqref{eq:tau2} is well-defined (possibly infinite). Denote the right-hand side of~\eqref{eq:tau2} by $f(t, x)$ ($t > 0$, $x \in \cop$). If $e_\xi(x) = e^{-\xi |x|}$, then for $\xi, z > 0$ we have, by Fubini,
\formula{
 \int_0^\infty \int_{-\infty}^\infty f(t, x) e^{-\xi |x|} e ^{-z t} dx dt & = \frac{1}{\pi} \int_0^\infty \frac{\cos \thet_\lambda \Psi'(\lambda) \Pi_\even e_\xi(\lambda)}{(z + \Psi(\lambda)) \Psi(\lambda)} \, d\lambda ;
}
the use of Fubini here is justified by the decomposition of the integrand into an absolutely integrable part and a non-negative part, as in~\eqref{eq:tau-decomposition}. Using $\cos \thet_\lambda = (1 + K_\lambda^2)^{-1/2}$ and formula~\eqref{eq:piexp}, we obtain that
\formula{
 \hspace*{5em} & \hspace*{-5em} \int_0^\infty \int_{-\infty}^\infty f(t, x) e^{-\xi |x|} e ^{-z t} dx dt \\
 & = \frac{2}{\pi \xi} \int_0^\infty \frac{\Psi'(\lambda)}{(z + \Psi(\lambda)) \Psi(\lambda)} \, \frac{K_\lambda L_\lambda(\xi) - K_\lambda(\xi)}{1 + K_\lambda^2} \, d\lambda .
}
By Lemma~\ref{lem:jump}, substitution $\Psi(\lambda) = s$ and Proposition~\ref{prop:cbf:repr}(b),
\formula{
 \hspace*{4em} & \hspace*{-4em} \int_0^\infty \int_{-\infty}^\infty f(t, x) e^{-\xi |x|} e ^{-z t} dx dt \\
 & = -\frac{2}{\pi \xi z} \int_0^\infty \expr{\frac{1}{\Psi(\lambda)} - \frac{1}{z + \Psi(\lambda)}} \, \im \frac{\ph^+(\xi, -\Psi(\lambda))}{\ph^+(-\Psi(\lambda))} \, \Psi'(\lambda) d\lambda \\
 & = -\frac{2}{\pi \xi z} \int_0^\infty \expr{\frac{1}{s} - \frac{1}{z + s}} \, \im \frac{\ph^+(\xi, -s)}{\ph^+(-s)} \, ds \\
 & = \frac{2}{\xi z} \expr{\lim_{\eps \to 0^+} \frac{\ph(\xi, \eps)}{\ph(\eps)} - \frac{\ph(\xi, z)}{\ph(z)}} .
}
Note that $\ph(\xi, z) \le \ph(z)$, so that the above expression is finite. In particular, $f(t, x)$ is finite for almost all $t > 0$, $x \in \cop$.

On the other hand, for $z > 0$ and $x \in \R$, denote $g_z(x) = \ex_x \exp(-z \tau_0)$. Then $g_z(x) = u_z(x) / u_z(0)$, where $u_z(x)$ is the resolvent (or $z$-potential) kernel for the operator $\A$ (see Preliminaries). Hence,
\formula{
 \fourier g_z(\xi) & = \frac{1}{u_z(0)} \, \frac{1}{\Psi(\xi) + z} \, , && z > 0 , \, \xi \in \R ,
}
where, by the Fourier inversion formula,
\formula{
 u_z(0) & = \frac{1}{\pi} \int_0^\infty \frac{1}{\Psi(\zeta) + z} \, d\zeta = \ph(z) .
}
Furthermore, for $z > 0$ and $x \in \cop$, 
\formula{
 \int_0^\infty \pr_x(\tau_0 \le t) e^{-z t} dt & = \frac{\ex_x \exp(-z \tau_0)}{z} = \frac{g_z(x)}{z} \, .
}
By Fubini and Plancherel's theorem,
\formula{
 \hspace*{4em} & \hspace*{-4em} \int_0^\infty \int_{-\infty}^\infty \pr_x(\tau_0 \le t) e^{-\xi |x|} e ^{-z t} dx dt \\
 & = \frac{1}{z} \int_{-\infty}^\infty g_z(x) e^{-\xi |x|} dx = \frac{2}{\pi z} \int_0^\infty \frac{\xi}{\xi^2 + \zeta^2} \, \fourier g_z(\zeta) d\zeta \\
 & = \frac{2}{\pi \xi z} \, \frac{1}{\ph(z)} \int_0^\infty \frac{\xi^2}{\xi^2 + \zeta^2} \, \frac{1}{\Psi(\zeta) + z} d\zeta = \frac{2}{\xi z} \, \frac{\ph(\xi, z)}{\ph(z)} \, .
}
In particular, we have
\formula{
 \hspace*{3em} & \hspace*{-3em} \int_0^\infty \int_{-\infty}^\infty \pr_x(\tau_0 < \infty) e^{-\xi |x|} e ^{-z t} dx dt \\
 & = \frac{1}{z} \int_{-\infty}^\infty \pr_x(\tau_0 < \infty) e^{-\xi |x|} dx \\
 & = \lim_{\eps \to 0^+} \expr{\frac{\eps}{z} \int_{-\infty}^\infty \int_0^\infty \pr_x(\tau_0 < t) e^{-\xi |x|} e^{-\eps t} dx dt} \\
 & = \lim_{\eps \to 0^+} \expr{\frac{2}{\xi z} \, \frac{\ph(\xi, \eps)}{\ph(\eps)}} ,
}
and so finally
\formula{
 \hspace*{8em} & \hspace*{-8em} \int_0^\infty \int_{-\infty}^\infty \pr_x(t < \tau_0 < \infty) e^{-\xi |x|} e ^{-z t} dx dt \\
 & = \frac{2}{\xi z} \expr{\lim_{\eps \to 0^+} \frac{\ph(\xi, \eps)}{\ph(\eps)} - \frac{\ph(\xi, z)}{\ph(z)}} \\
 & = \int_0^\infty \int_{-\infty}^\infty f(t, x) e^{-\xi |x|} e ^{-z t} dx dt .
}
By the uniqueness of the Laplace transform, formula~\eqref{eq:tau2} follows for almost every pair $t > 0$, $x \in \cop$. By dominated convergence, both sides of~\eqref{eq:tau2} are continuous in $t > 0$, and the theorem is proved.
\end{proof}

\begin{remark}
\label{rem:misproof}
One could try the following argument for the proof of Theorem~\ref{th:tau}: By monotone convergence,
\formula{
 \pr_x(\tau_0 > t) & = \int_{-\infty}^\infty p^{\cop}_t(x, y) dy \\
 & = \lim_{\xi \to 0^+} \int_{-\infty}^\infty p^{\cop}_t(x, y) e^{-\xi |y|} dy = \lim_{\xi \to 0^+} P^{\cop}_t e_\xi(x) ,
}
where $e_\xi(x) = e^{-\xi |x|}$. For each $\xi > 0$, by Theorem~\ref{th:spectral} we have
\formula[eq:misproof]{
\begin{aligned}
 P^{\cop}_t e_\xi(x) & = \frac{1}{\pi} \, \Pi_\even^* (e^{-t \Psi} \Pi_\even e_\xi)(x) \\
 & = \frac{1}{\pi} \int_0^\infty F_\lambda(x) e^{-t \Psi(\lambda)} \Pi_\even e_\xi(\lambda) d\lambda .
\end{aligned}
}
Recall that as $\xi \to 0^+$, $\Pi_\even e_\xi(\lambda)$ converges to $\cos \thet_\lambda \Psi'(\lambda) / \Psi(\lambda)$. Hence, if the above integrals were uniformly integrable as $\xi \to 0^+$, we would obtain 
\formula{
 \pr_x(\tau_0 > t) & = \frac{1}{\pi} \int_0^\infty F_\lambda(x) e^{-t \Psi(\lambda)} \, \frac{\cos \thet_\lambda \Psi'(\lambda)}{\Psi(\lambda)} \, d\lambda .
}
However, if $\pr_x(\tau_0 < \infty) < 1$ (that is, when $X_t$ is transient), we know by Theorem~\ref{th:tau} that this formula is false! This proves that the integrand in~\eqref{eq:misproof} in some cases fails to be uniformly integrable, and its limit (in the sense of distributions) may contain a point mass at $0$. For this reason, a more direct approach to the proof of Theorem~\ref{th:tau}, similar to the argument sketched above, seems to be problematic and require stronger regularity conditions on $\Psi$.\qed
\end{remark}

%
%                            ---------- o ----------
%

\section{Examples}
\label{sec:examples}

In this section we consider a few examples: processes frequently found in literature and an irregular case. We focus on spectral theory, that is, on Theorems~\ref{th:flambda}, \ref{th:glambda} and~\ref{th:spectral}, and only record formulae given in Corollary~\ref{cor:pdt} and~\ref{th:tau}. As mentioned in the introduction, detailed properties of first hitting times to points will be studied in a separate article.

Plots have been prepared using \emph{Mathematica 8.1} at the Wroc{\l}aw University of Technology. Numerical integration was used for the computation of $\thet_\lambda$ and fast Fourier transform for the calculation of $G_\lambda(x)$. 

\begin{figure}
\centering
\begin{tabular}{cc}
\includegraphics[width=0.49\textwidth]{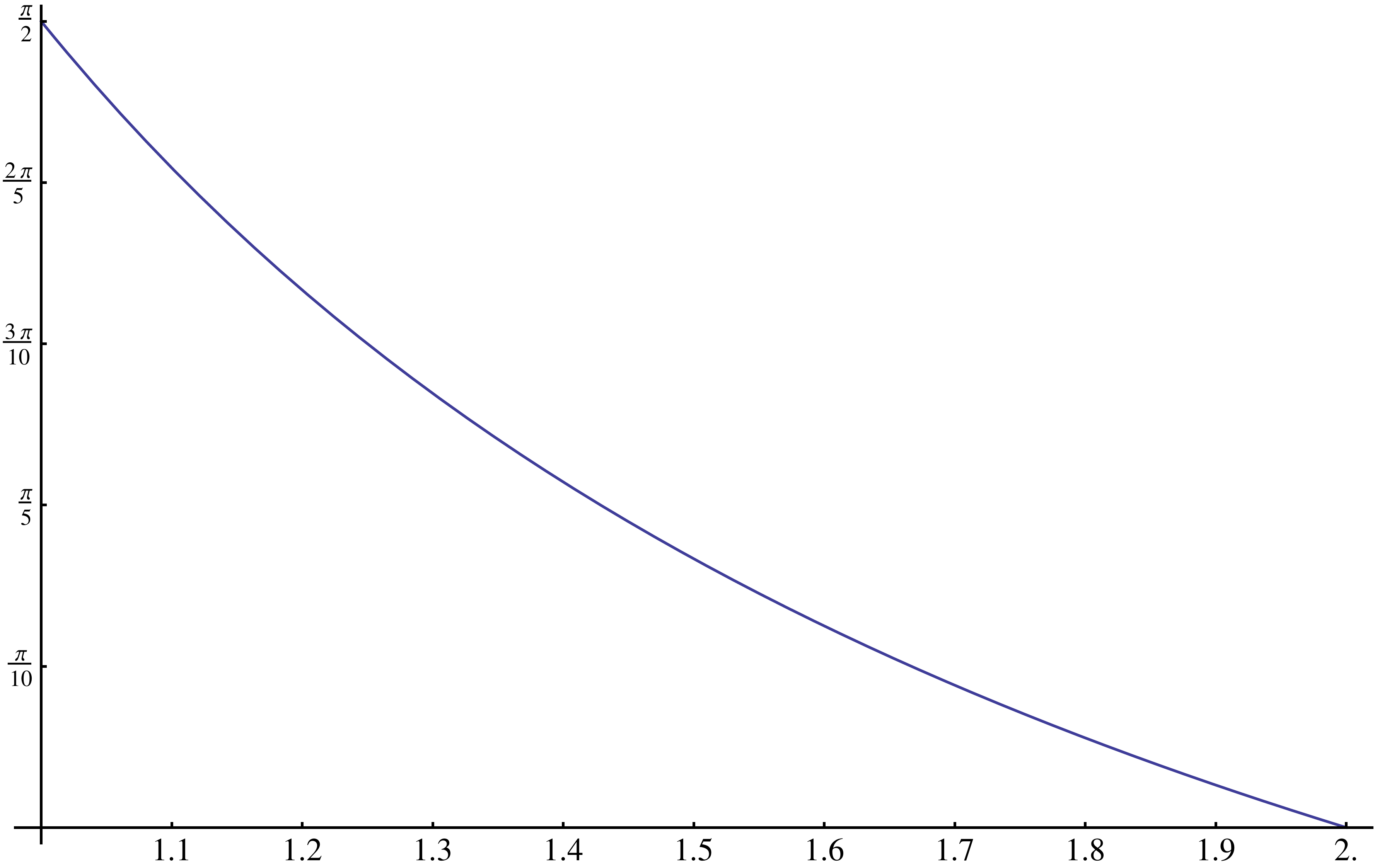}&
\includegraphics[width=0.49\textwidth]{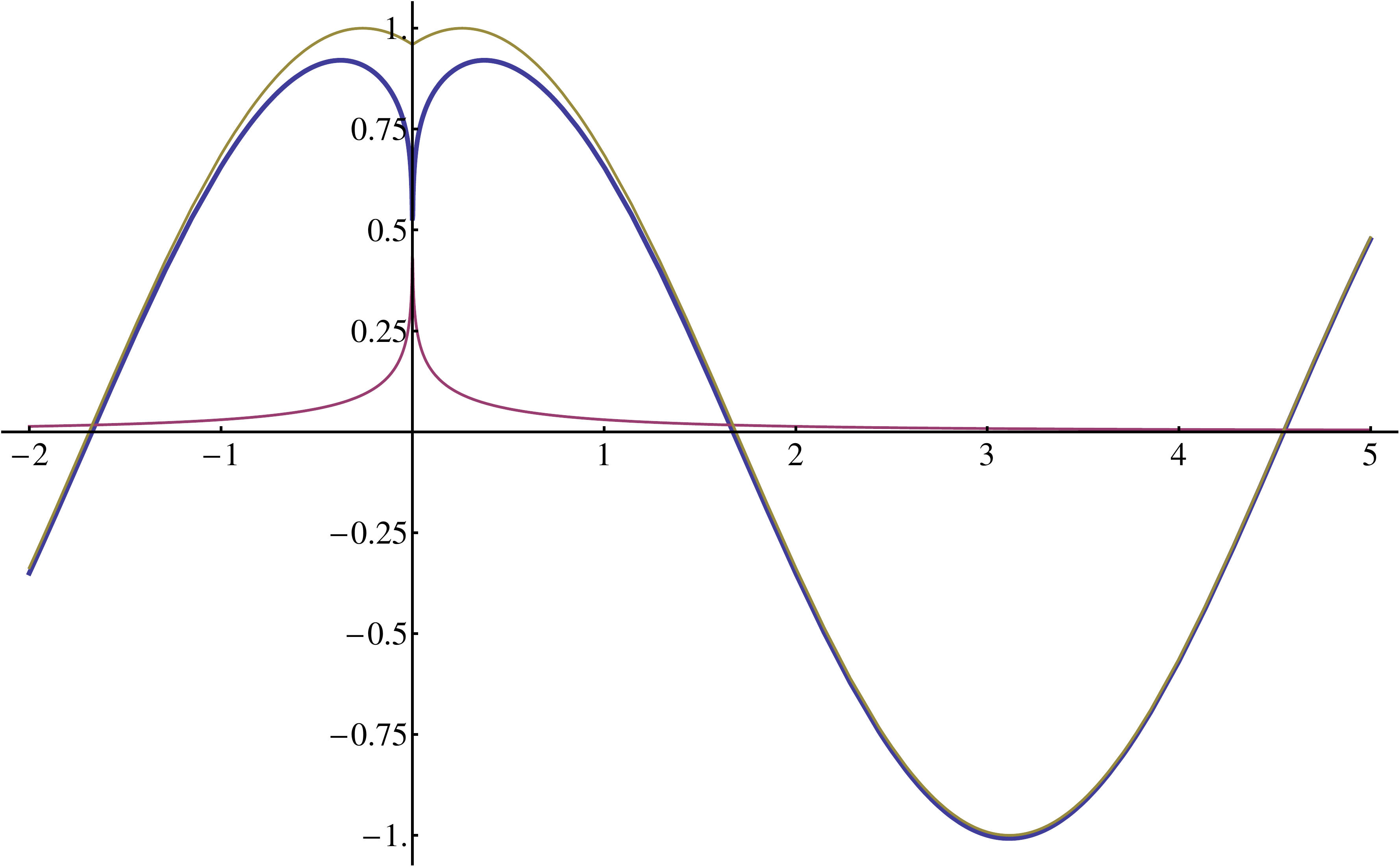}\\
\footnotesize $\thet_\lambda(\alpha) = \pi/\alpha - \pi/2$&
\footnotesize \parbox{0.35\textwidth}{\centering $F_1(x)$ (blue), $G_1(x)$ (purple) and $\sin(|x| + \thet_1)$ (brown) for $\alpha = 1.1$}\\[0.5em]
\includegraphics[width=0.49\textwidth]{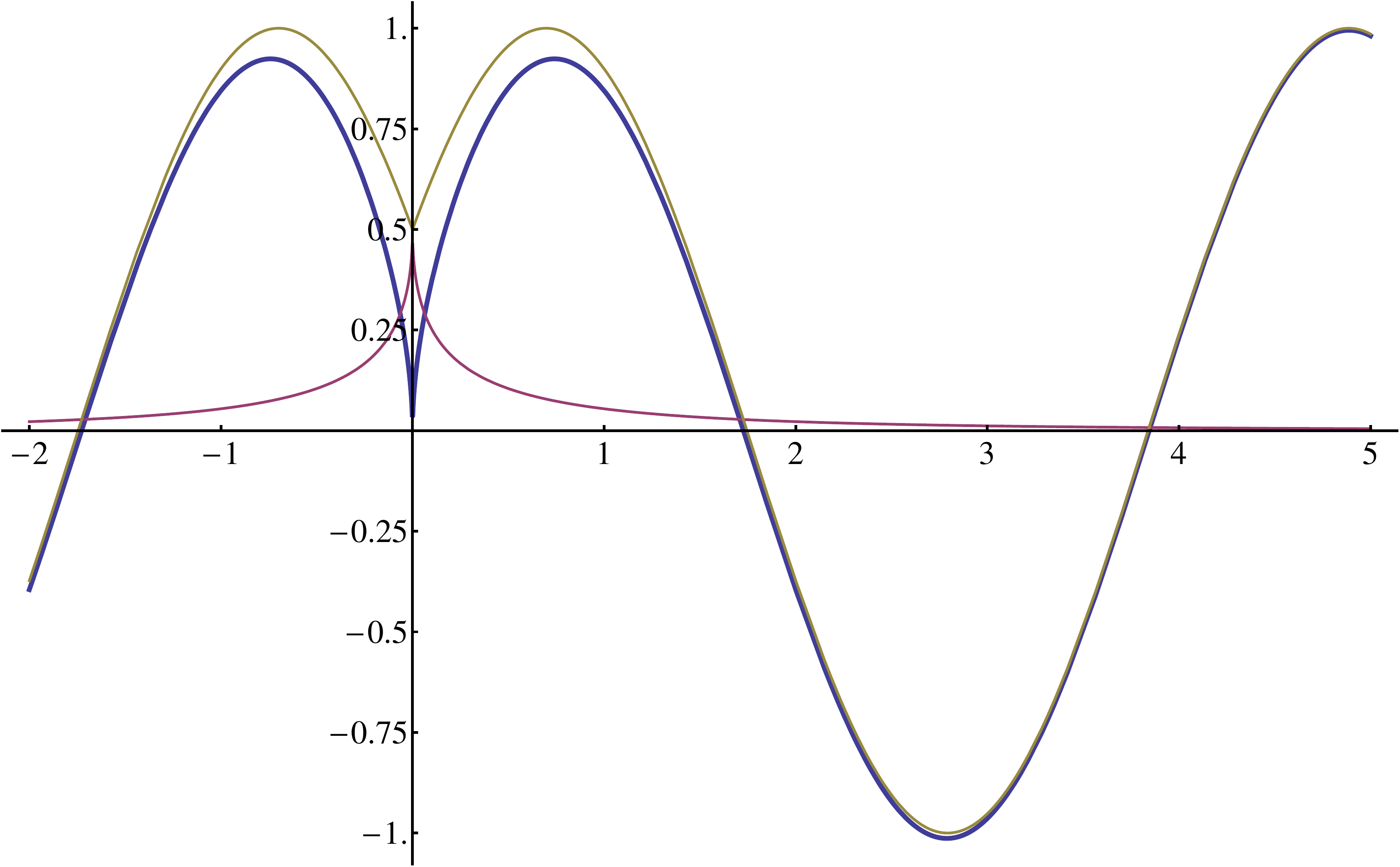}&
\includegraphics[width=0.49\textwidth]{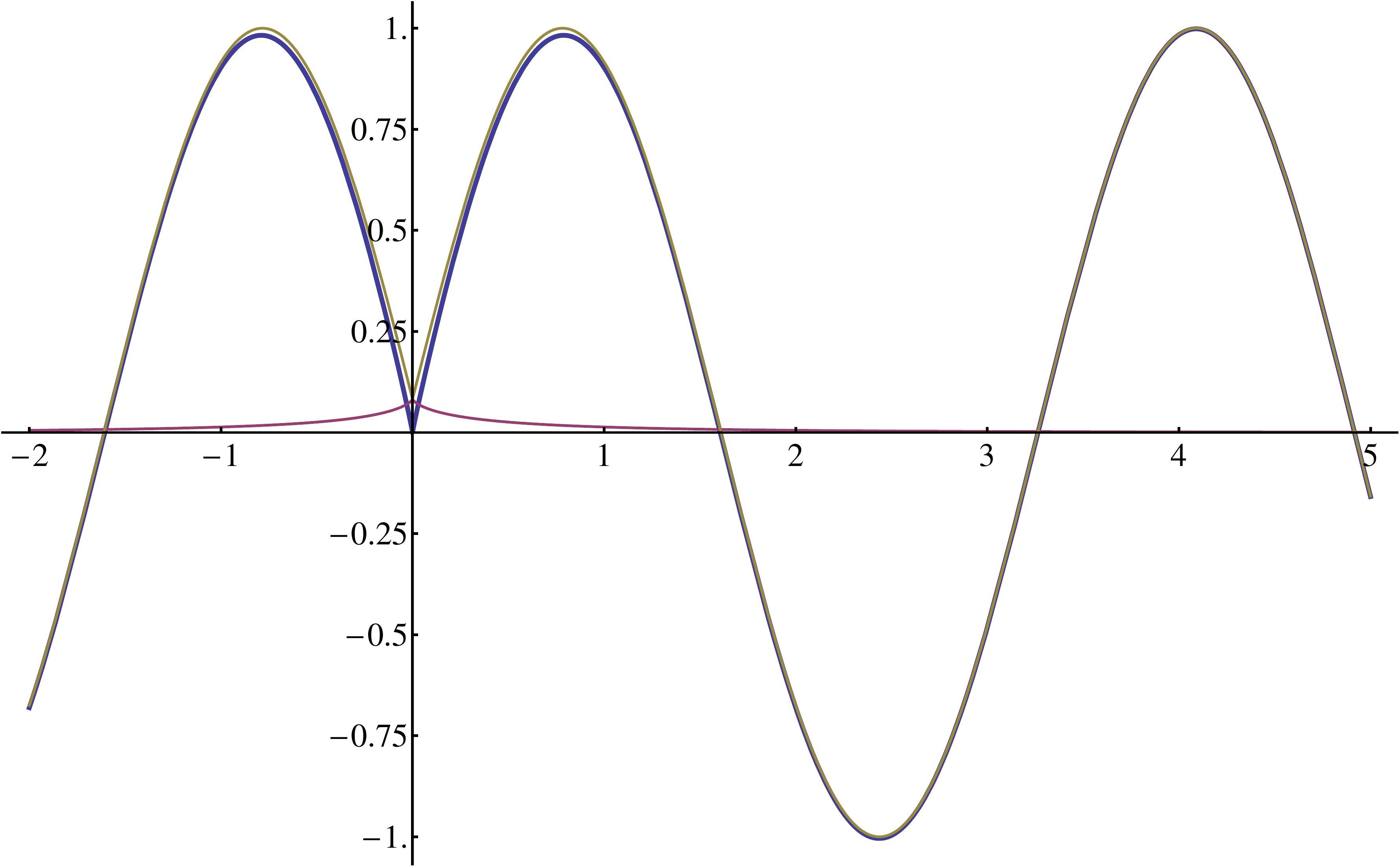}\\
\footnotesize \parbox{0.35\textwidth}{\centering $F_1(x)$ (blue), $G_1(x)$ (purple) and $\sin(|x| + \thet_1)$ (brown) for $\alpha = 1.5$}&
\footnotesize \parbox{0.35\textwidth}{\centering $F_1(x)$ (blue), $G_1(x)$ (purple) and $\sin(|x| + \thet_1)$ (brown) for $\alpha = 1.9$}
\end{tabular}
\caption{Plots for $\Psi(\xi) = |\xi|^\alpha$.}
\label{fig:2}
\end{figure}

\begin{example}
Let $\Psi(\xi) = |\xi|^\alpha$, $\alpha \in (1, 2]$, be the L\'evy-Khintchine exponent of the one-dimensional symmetric $\alpha$-stable L\'evy process. The distribution of the hitting times to points for these processes have been previously studied in~\cite{bib:yyy09}.

By substituting $\xi = \lambda s$ in the integral over $(0, \lambda)$ and $\xi = \lambda / s$ in the integral over $(\lambda, \infty)$, we obtain (see~\eqref{eq:k0})
\formula{
 K_\lambda & = -\frac{1}{\pi} \, \pvint_0^\infty \frac{\alpha \lambda^{\alpha - 1}}{\lambda^\alpha - \xi^\alpha} \, d\xi = -\frac{\alpha}{\pi} \int_0^1 \expr{\frac{1}{1 - s^\alpha} - \frac{1}{s^{2 - \alpha} (1 - s^\alpha)}} \, ds .
}
By series expansion and Fubini,
\formula{
 K_\lambda & = \frac{\alpha}{\pi} \sum_{n = 0}^\infty \expr{\int_0^1 (s^{\alpha - 2} - 1) s^{n \alpha} ds} = \frac{\alpha}{\pi} \sum_{n = 0}^\infty \expr{\frac{1}{(n + 1) \alpha - 1} - \frac{1}{n \alpha + 1}} .
}
Recall that for $z \in (0, \pi)$,
\formula{
 \pi \cot (\pi z) & = \frac{1}{z} + \sum_{n = 1}^\infty \expr{\frac{1}{z - n} + \frac{1}{z + n}} .
}
Hence, after some rearrangement of the sum,
\formula{
 K_\lambda & = -\frac{1}{\pi} \sum_{n = 0}^\infty \expr{\frac{1}{1/\alpha + n} + \frac{1}{1/\alpha - (n + 1)}} \\
 & = - \frac{1}{\pi} \expr{\frac{1}{1/\alpha} + \sum_{k = 1}^\infty \expr{\frac{1}{1/\alpha + k} + \frac{1}{1/\alpha - k}}} = -\cot \frac{\pi}{\alpha} \, .
}
Hence,
\formula[eq:stable:theta]{
 \thet_\lambda & = \arctan(K_\lambda) = \frac{\pi}{\alpha} - \frac{\pi}{2} \, ,
}
and $\cos \thet_\lambda = \sin(\pi / \alpha)$, $\sin(\thet_\lambda) = -\cos(\pi / \alpha)$. Furthermore, by~\eqref{eq:gintegral},
\formula{
 G_\lambda(x) & = \frac{\alpha \lambda^{\alpha - 1} \sin(\pi \alpha / 2) \sin(\pi / \alpha)}{\pi} \times \\
 & \hspace*{7em} \times \int_0^\infty \frac{\xi^\alpha}{\lambda^{2 \alpha} - 2 \lambda^\alpha \xi^\alpha \cos(\pi \alpha / 2) + \xi^{2 \alpha}} \, e^{-\xi |x|} d\xi \\
 & = \frac{\alpha \sin(\pi \alpha) \sin(\pi / \alpha)}{\pi} \, \int_0^\infty \frac{s^\alpha}{1 - 2 s^\alpha \cos(\pi \alpha) + s^{2 \alpha}} \, e^{-\lambda s |x|} ds .
}
Hence,
\formula[eq:stable:f]{
\begin{aligned}
 F_\lambda(x) & = \sin \expr{|\lambda x| + \frac{\pi}{\alpha} - \frac{\pi}{2}} - \frac{\alpha \sin(\pi \alpha / 2) \sin(\pi / \alpha)}{\pi} \times \\
 & \hspace*{10em} \times \int_0^\infty \frac{s^\alpha}{1 - 2 s^\alpha \cos(\pi \alpha / 2) + s^{2 \alpha}} \, e^{-\lambda s |x|} ds .
\end{aligned}
}
Note that $F_\lambda(x) = F_1(\lambda x)$, which could be deduced a priori using scaling properties of $X_t$. As expected, in the limiting case $\alpha \to 1^+$ we have $\thet_\lambda \to \pi/2$ and $G_\lambda(x) \to 0$ (the Cauchy process never hits $0$, hence $F_\lambda(x) = \cos(\lambda x) = \sin(|\lambda x| + \pi/2)$ is the even eigenfunction in $\cop$). Also, for $\alpha = 2$, we obtain $\thet_\lambda = 0$ and $G_\lambda(x) = 0$ ($F_\lambda(x) = \sin |\lambda x|$ is the even eigenfunction for the Brownian motion in $\cop$). See Figure~\ref{fig:2} for plots of $F_\lambda$.

Since $X_t$ is recurrent, we have $\pr_x(\tau_0 = \infty) = 0$. Hence, by Theorem~\ref{th:tau} and scaling,
\formula{
 \pr_x(\tau_0 > t) & = \frac{\alpha \sin(\pi / \alpha)}{\pi} \int_0^\infty \frac{e^{-t \lambda^\alpha} F_1(\lambda x)}{\lambda} \, d\lambda .
}
By Corollary~\ref{cor:pdt},
\formula{
 p^\cop_t(x, y) & = \frac{1}{\pi} \int_0^\infty e^{-t \lambda^\alpha} (F_1(\lambda x) F_1(\lambda y) + \sin(\lambda x) \sin(\lambda y)) d\lambda .
}
These formulae are applicable for numerical computation.
\end{example}

\begin{example}
Let $\alpha \in (0, 2)$, $\beta > 0$, and $\Psi(\xi) = \xi^2 + \beta |\xi|^\alpha$. The corresponding L\'evy process $X_t$ is a mixture of Brownian motion and the symmetric $\alpha$-stable process. Note that $X_t$ is recurrent if and only if $\alpha > 1$. We have
\formula{
 \thet_\lambda & = \arctan\expr{-\frac{1}{\pi} \, \pvint_0^\infty \frac{2 \lambda + \alpha \beta \lambda^{\alpha - 1}}{\lambda^2 + \beta \lambda^\alpha - \xi^2 - \beta \xi^\alpha} \, d\xi} .
}
This integral does not have a closed form, but it can be proved that $\lim_{\lambda \to \infty} \thet_\lambda = 0$ and $\lim_{\lambda \to 0^+} \thet_\lambda = \min(\pi/\alpha - \pi/2, \pi/2)$. By Theorem~\ref{th:glambda},
\formula{
 F_\lambda(x) & = \sin(|\lambda x| + \thet_\lambda) - \frac{(2 \lambda + \alpha \beta \lambda^{\alpha - 1}) \cos \thet_\lambda}{\pi} \times \\
 & \qquad \times \int_0^\infty \frac{\sin(\alpha \pi / 2) \xi^\alpha}{(\lambda^2 + \beta \lambda^\alpha - \xi^2)^2 + \xi^{2 \alpha} - 2 \cos(\alpha \pi / 2) (\lambda^2 + \beta \lambda^\alpha - \xi^2) \xi^\alpha} \, e^{-\xi |x|} d\xi .
}
\end{example}

\begin{example}
Suppose that $\Psi(\xi) = (\xi^2 + 1)^{\alpha/2} - 1$ is the L\'evy-Khintchine exponent of the relativistic $\alpha$-stable L\'evy process, $\alpha \in (1, 2)$, corresponding to unit mass. In this case again
\formula{
 \thet_\lambda & = \arctan\expr{-\frac{1}{\pi} \pvint_0^\infty \frac{\alpha \lambda (\lambda^2 + 1)^{\alpha/2 - 1}}{(\lambda^2 + 1)^{\alpha/2} - (\xi^2 + 1)^{\alpha/2}} \, d\xi}
}
cannot be evaluated symbolically. However, with some effort, one can show that $\lim_{\lambda \to 0^+} \thet_\lambda = 0$ and $\lim_{\lambda \to \infty} \thet_\lambda = \pi/\alpha - \pi/2$. We also have
\formula{
 F_\lambda(x) & = \sin(|\lambda x| + \thet_\lambda) - \frac{\alpha \lambda (\lambda^2 + 1)^{\alpha/2 - 1} \cos \thet_\lambda}{\pi} \times \\
 & \qquad \times \int_1^\infty \frac{\sin(\alpha \pi/2) (1 - \xi^2)^{\alpha/2}}{(\lambda^2 + 1)^\alpha + (1 - \xi^2)^\alpha - 2 \cos(\alpha \pi/2) (\lambda^2 + 1)^{\alpha/2} (1 - \xi^2)^{\alpha/2}} e^{-\xi |x|} d\xi .
}
In particular, $F_\lambda(x) - \sin(|\lambda x| + \thet_\lambda)$ decays exponentially fast as $|x| \to \infty$.
\end{example}

\begin{figure}
\centering
\begin{tabular}{cc}
\includegraphics[width=0.49\textwidth]{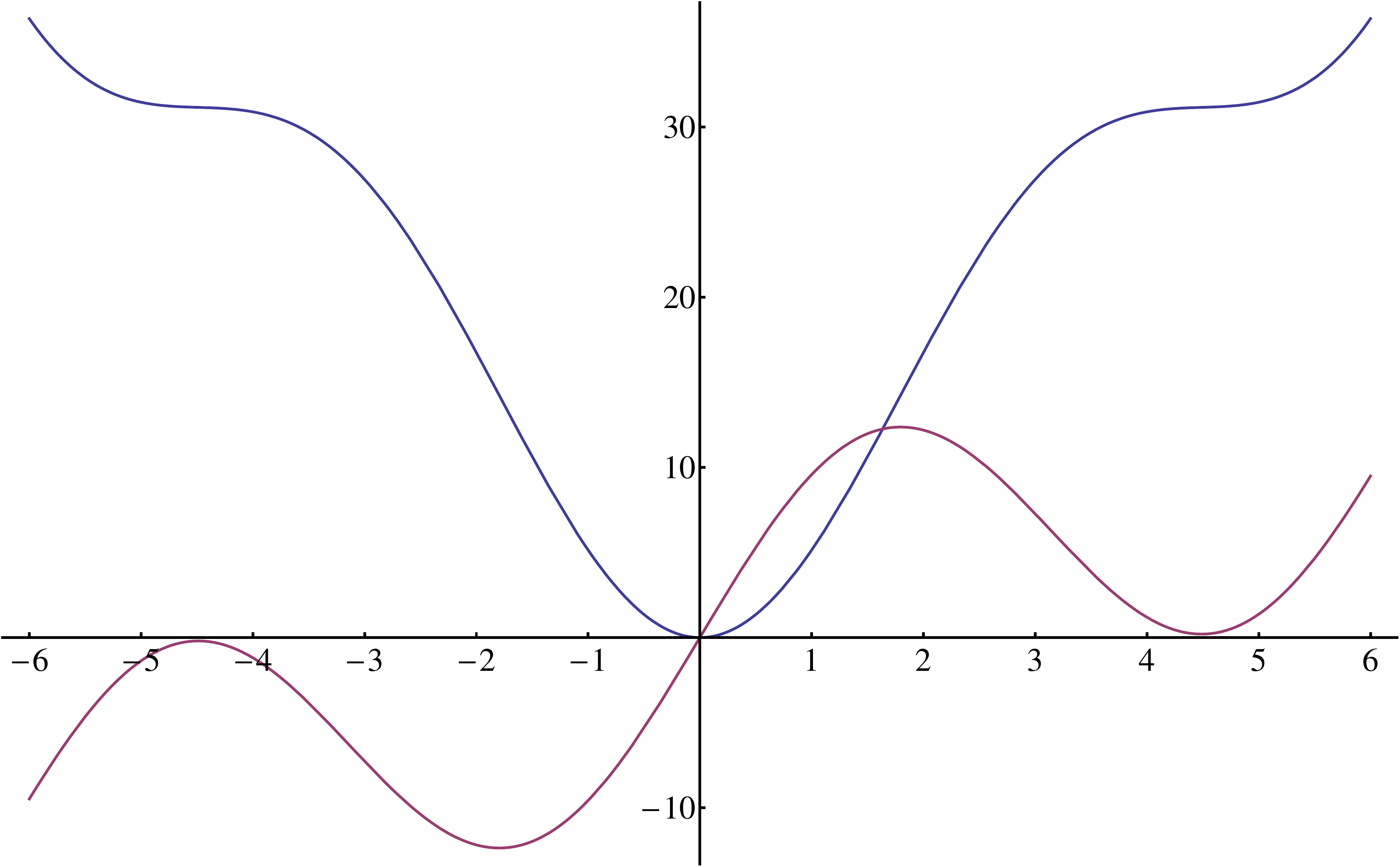}&
\includegraphics[width=0.49\textwidth]{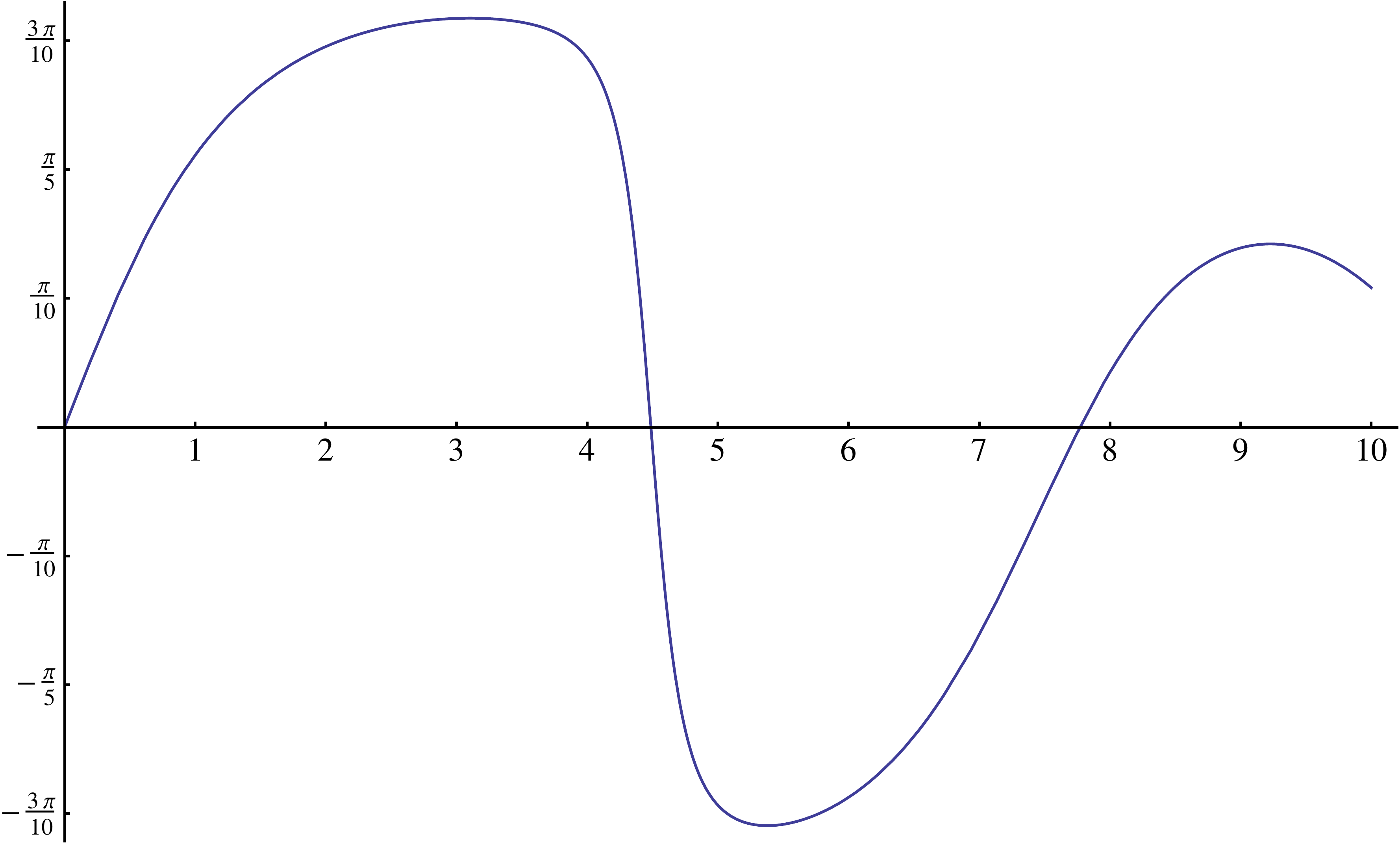}\\
\footnotesize $\Psi(\xi)$ (blue) and $\Psi'(\xi)$ (purple)&
\footnotesize $\thet_\lambda$\\[0.5em]
\includegraphics[width=0.49\textwidth]{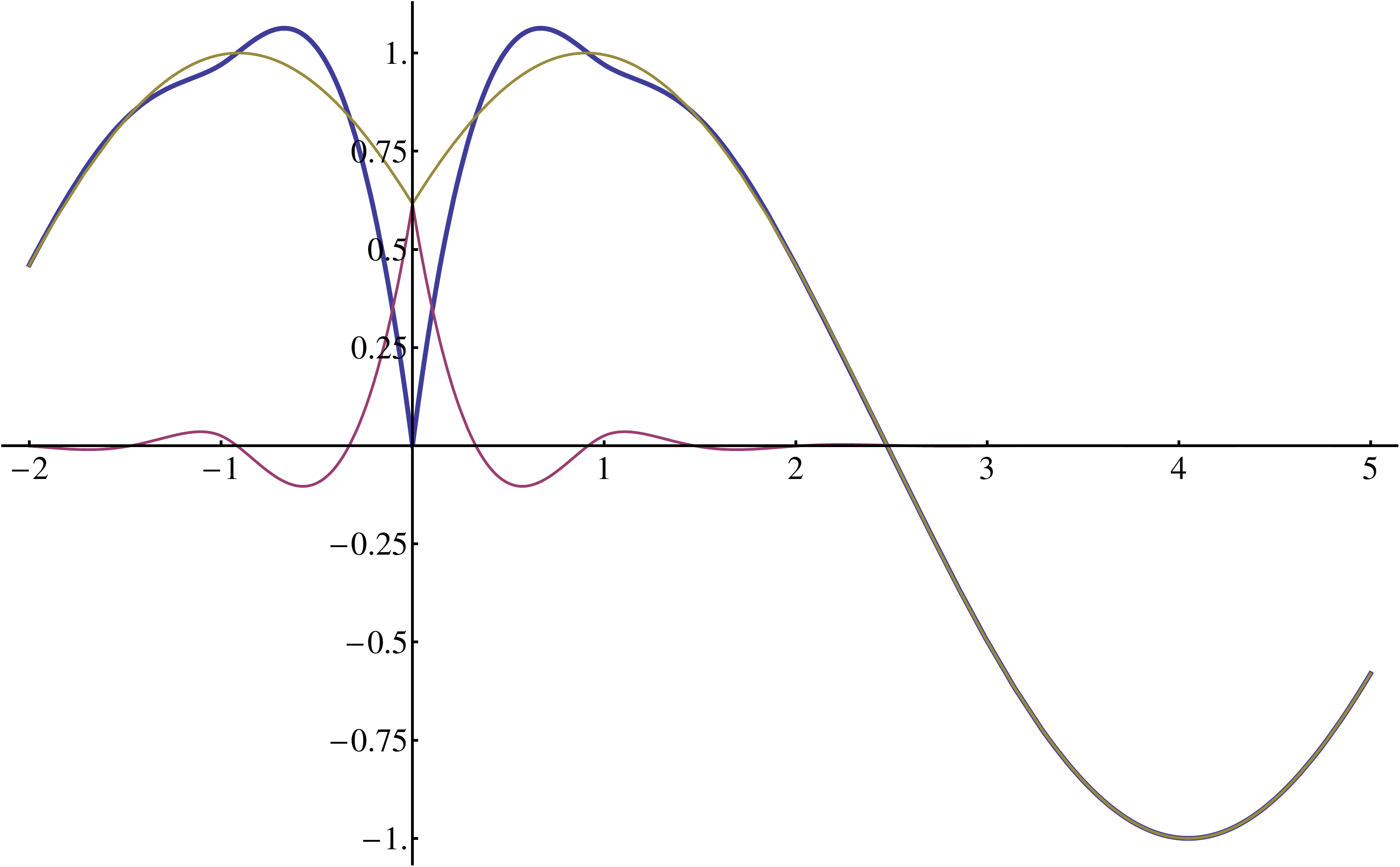}&
\includegraphics[width=0.49\textwidth]{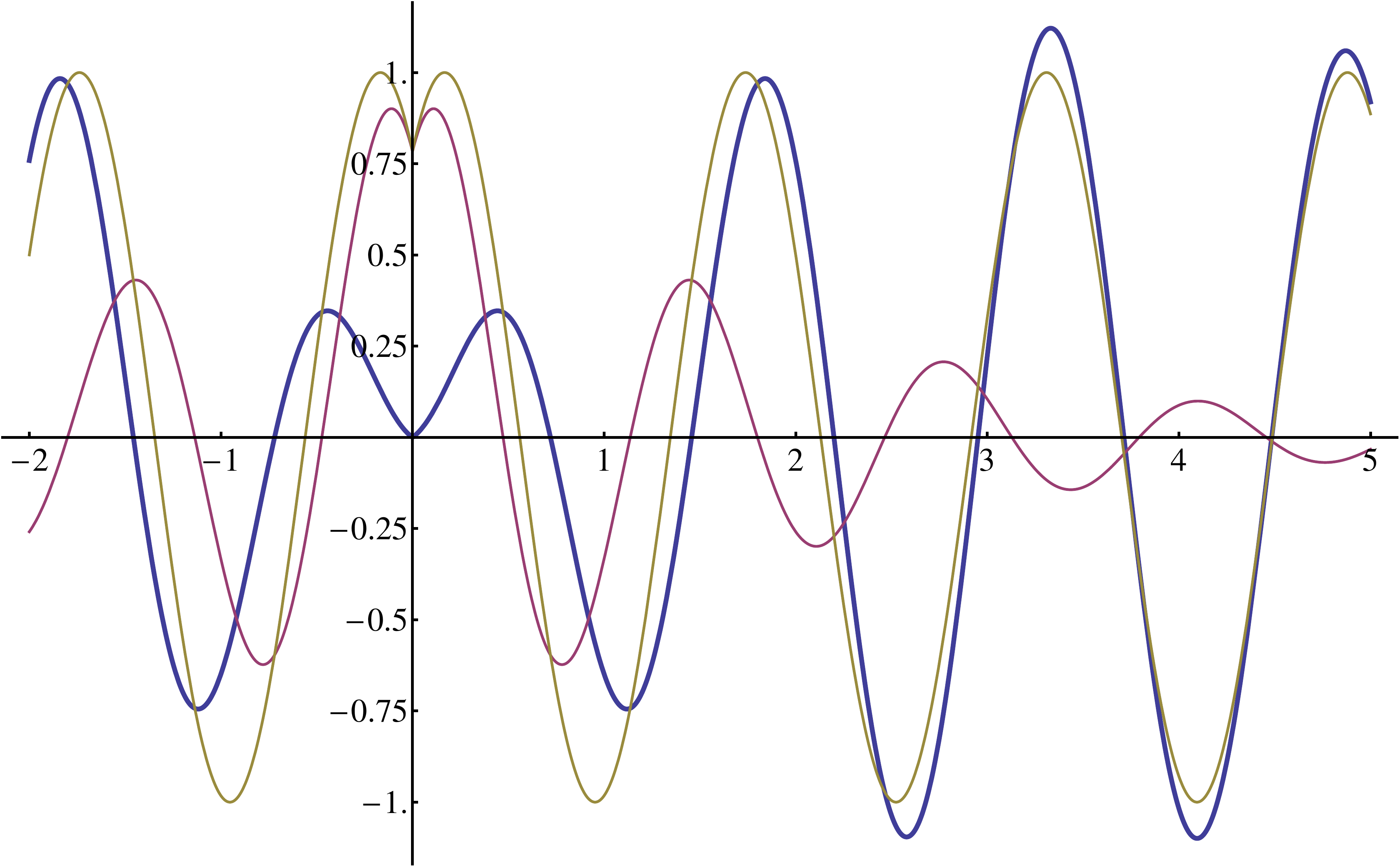}\\
\footnotesize \parbox{0.35\textwidth}{\centering $F_1(x)$ (blue), $G_1(x)$ (purple) and $\sin(|x| + \thet_1)$ (brown)}&
\footnotesize \parbox{0.35\textwidth}{\centering $F_4(x)$ (blue), $G_4(x)$ (purple) and $\sin(4|x| + \thet_4)$ (brown)}\\[0.5em]
\includegraphics[width=0.49\textwidth]{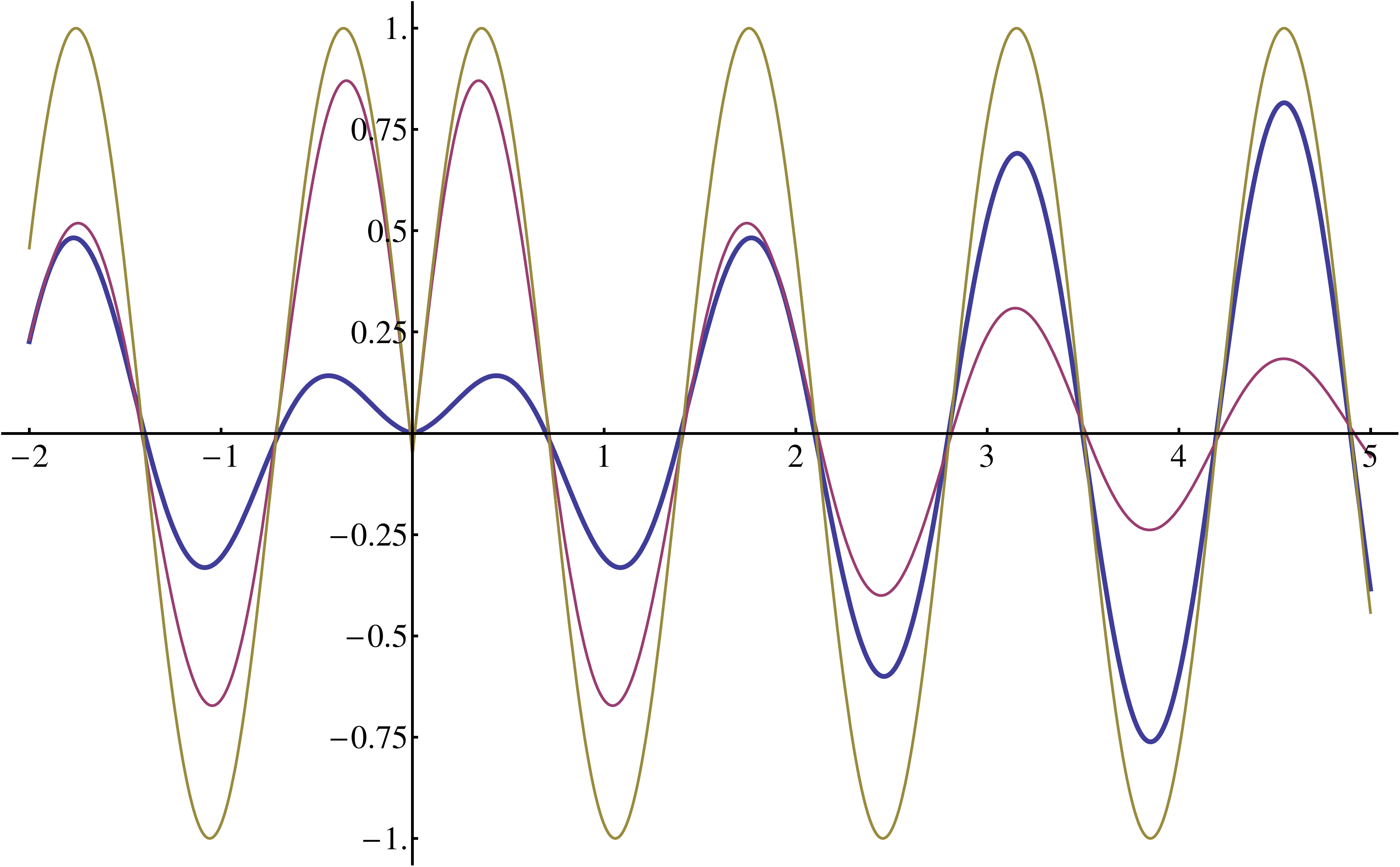}&
\includegraphics[width=0.49\textwidth]{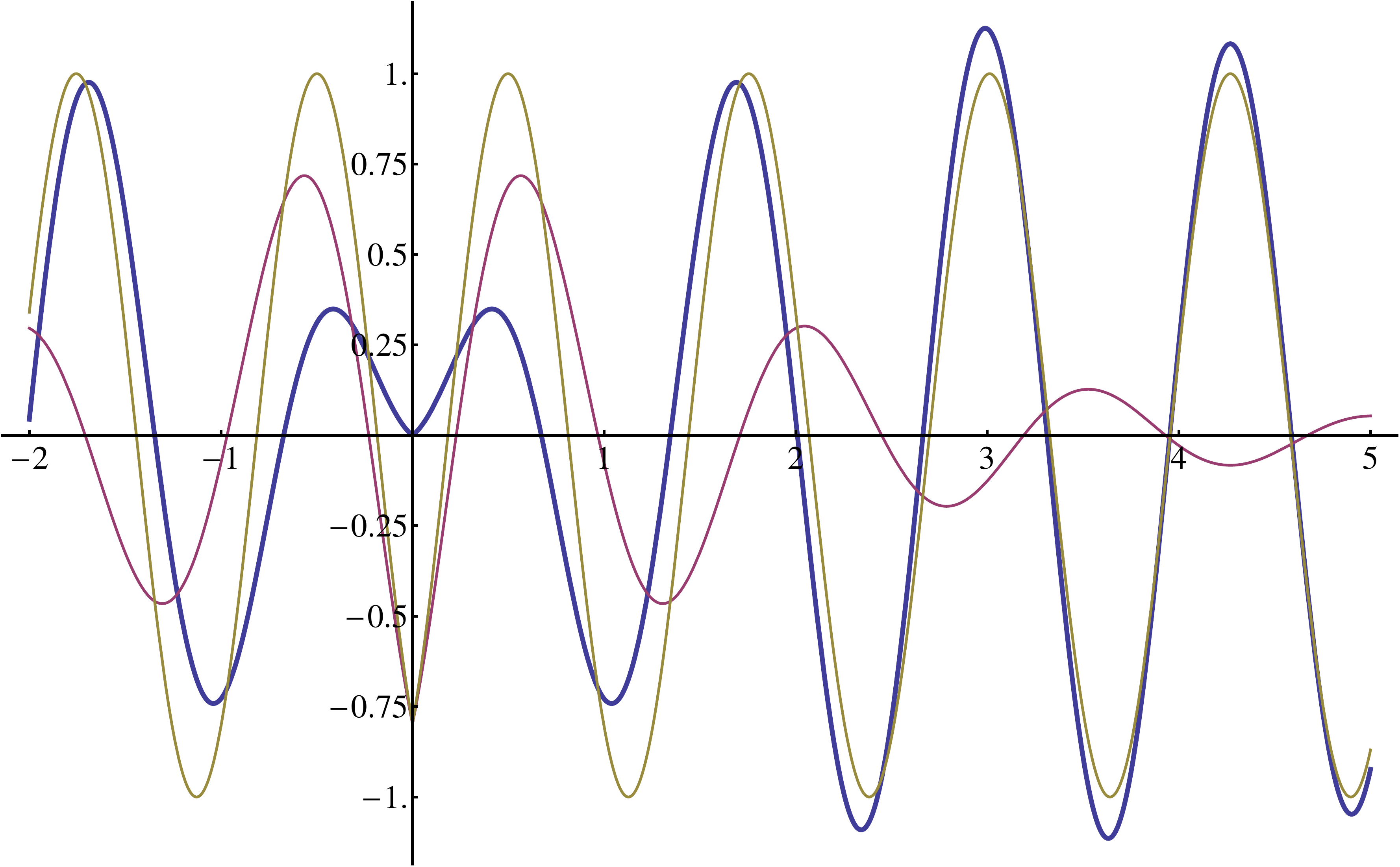}\\
\footnotesize \parbox{0.35\textwidth}{\centering $F_{4.5}(x)$ (blue), $G_{4.5}(x)$ (purple) and $\sin(4.5 |x| + \thet_{4.5})$ (brown)}&
\footnotesize \parbox{0.35\textwidth}{\centering $F_5(x)$ (blue), $G_5(x)$ (purple) and $\sin(5|x| + \thet_5)$ (brown)}
\end{tabular}
\caption{Plots for $\Psi(\xi) = \xi^2 + 9 (1 - \cos \xi)$. Note the differences between $\lambda = 1$ ($\thet_\lambda > 0$, $G_\lambda$ attains global maximum at $0$), $\lambda = 4$ ($\thet_\lambda > 0$, $G_\lambda$ has a positive local minimum at $0$, $\lambda = 4.5$ ($\thet_\lambda$ and $G_\lambda(0)$ close to $0$, slow convergence of $F_\lambda(x)$ to $\sin(|\lambda x| + \thet_\lambda)$ as $|x| \to \infty$) and $\lambda = 5$ ($\thet_\lambda < 0$, $G_\lambda$ has a negative global minimum at $0$).}
\label{fig:1}
\end{figure}

\begin{example}
Let $\Psi(\xi) = \xi^2 + 9 (1 - \cos \xi)$. The corresponding process $X_t$ is the sum of the Brownian motion (with variance $2t$) and an independent compound Poisson process with jumps $\pm 1$ occurring at rate $9$. It can be easily checked that $\Psi'(\xi) > 0$ for $\xi > 0$, and therefore Theorem~\ref{th:flambda} applies. However, $\thet_\lambda$ is negative for some $\lambda > 0$, and $G_\lambda(x)$ is a signed function, see Figure~\ref{fig:1}.

Since $\psi(\xi) = \Psi(\sqrt{\xi}) = \xi + 9 (1 - \cos \sqrt{\xi})$ is not concave, Theorem~\ref{th:tau} (or Theorem~\ref{th:taux}) cannot be applied to $X_t$.
\end{example}

\begin{figure}
\centering
\begin{tabular}{cc}
\includegraphics[width=0.49\textwidth]{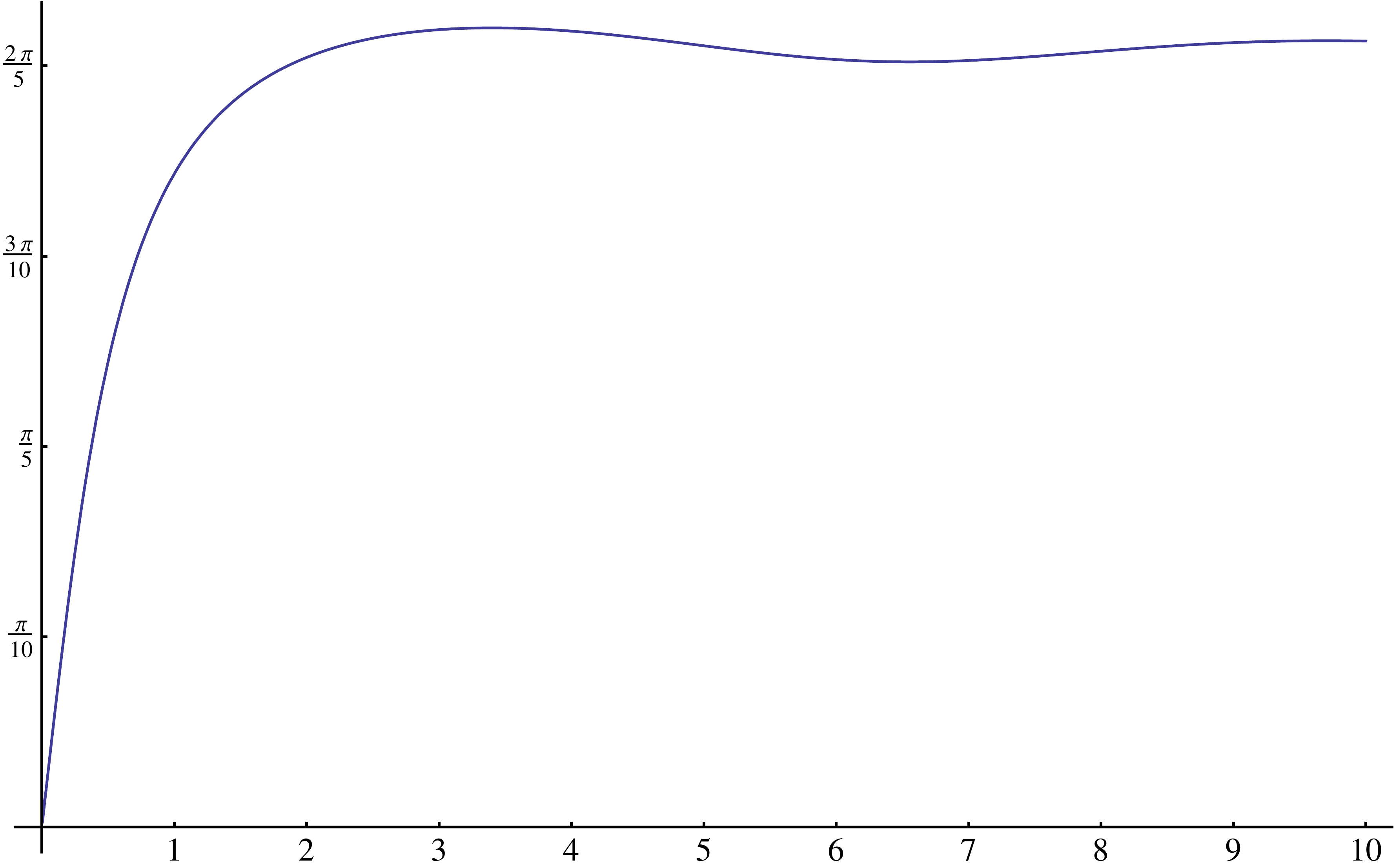}&
\includegraphics[width=0.49\textwidth]{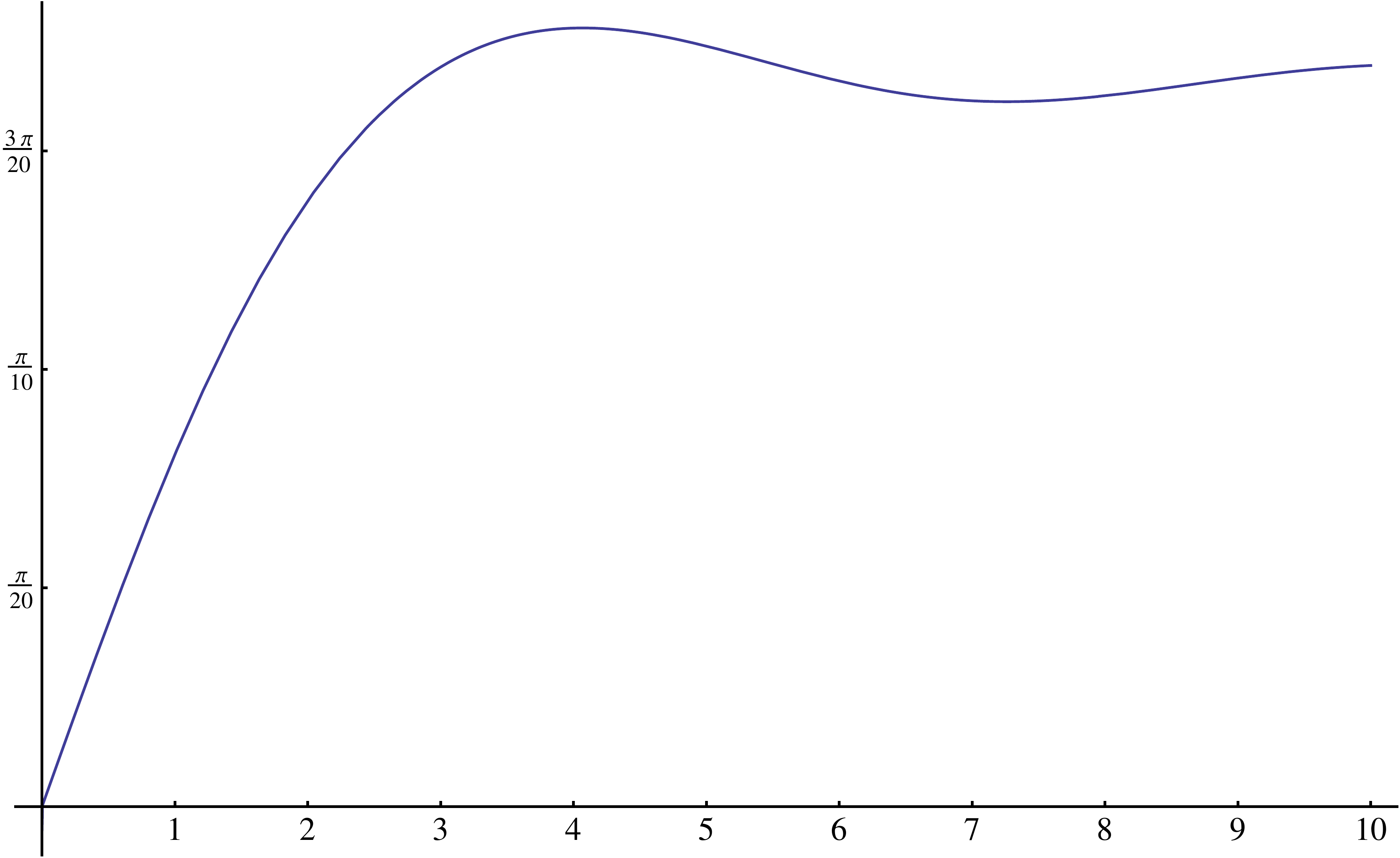}\\
\footnotesize $\thet_\lambda$ ($\alpha = 1.1$)&
\footnotesize $\thet_\lambda$ ($\alpha = 1.5$)\\[0.5em]
\includegraphics[width=0.49\textwidth]{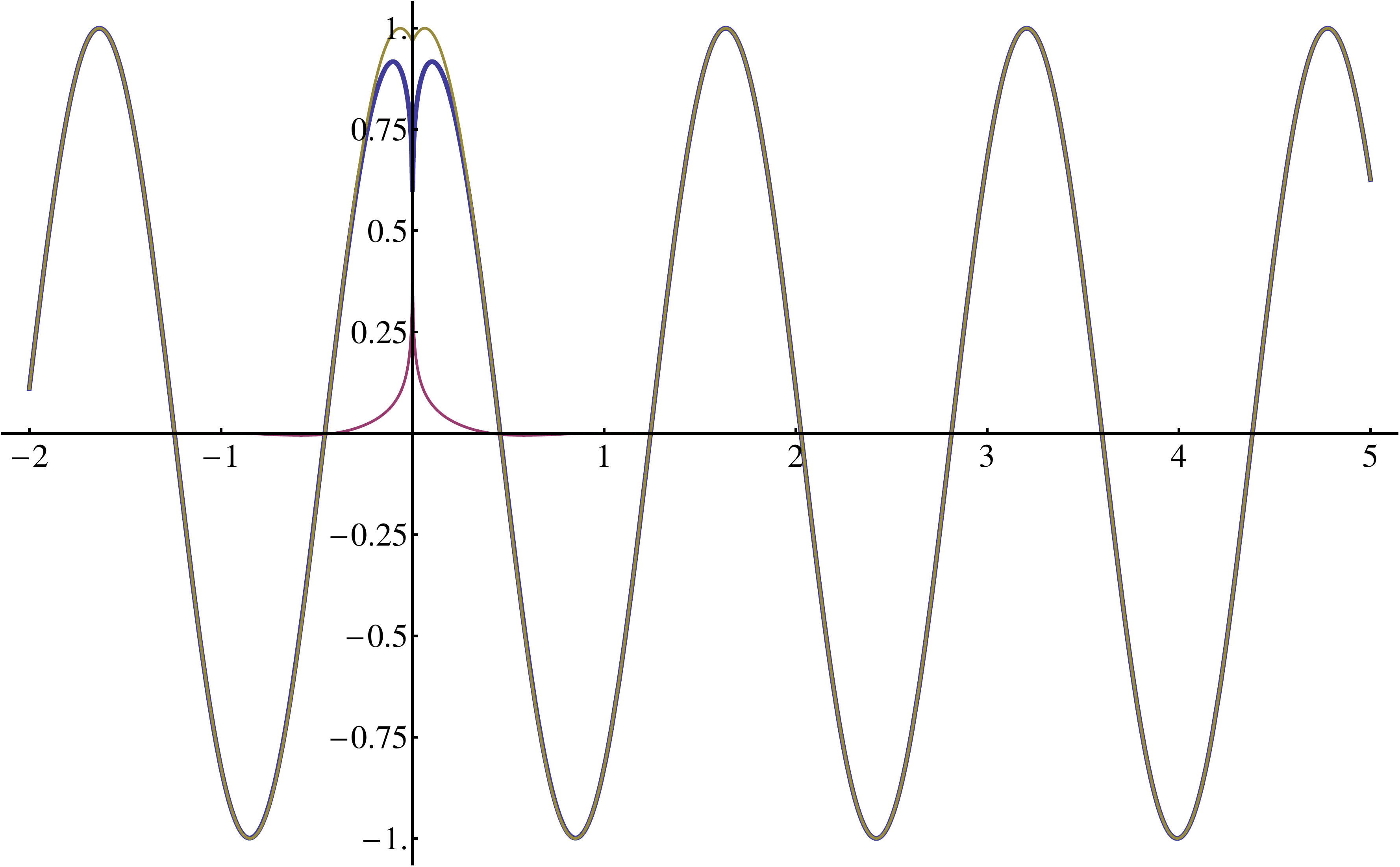}&
\includegraphics[width=0.49\textwidth]{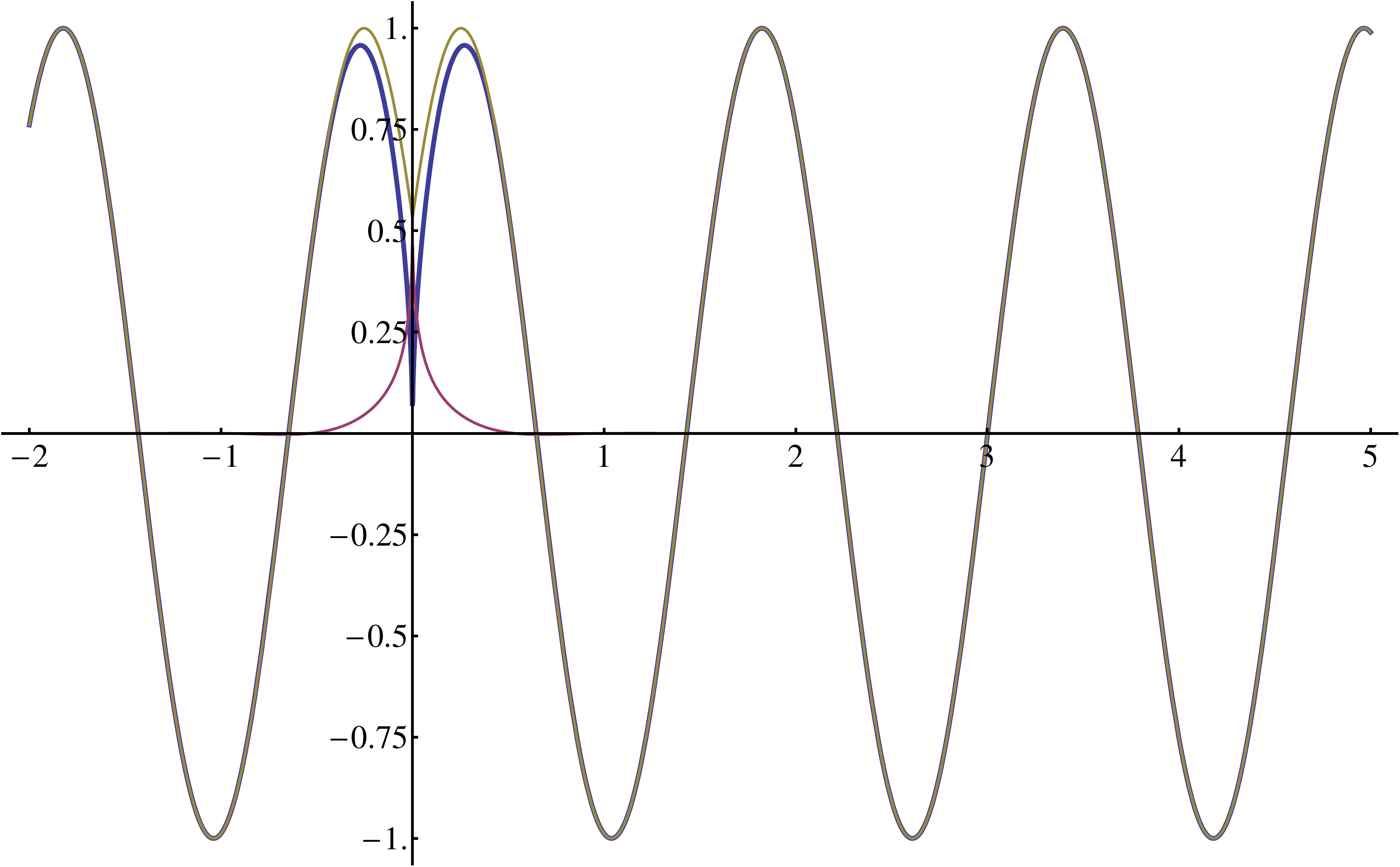}\\
\footnotesize \parbox{0.35\textwidth}{\centering $F_4(x)$ (blue), $G_1(x)$ (purple) and $\sin(4 |x| + \thet_4)$ (brown) for $\alpha = 1.1$}&
\footnotesize \parbox{0.35\textwidth}{\centering $F_4(x)$ (blue), $G_1(x)$ (purple) and $\sin(4 |x| + \thet_4)$ (brown) for $\alpha = 1.5$}\\[0.5em]
\includegraphics[width=0.49\textwidth]{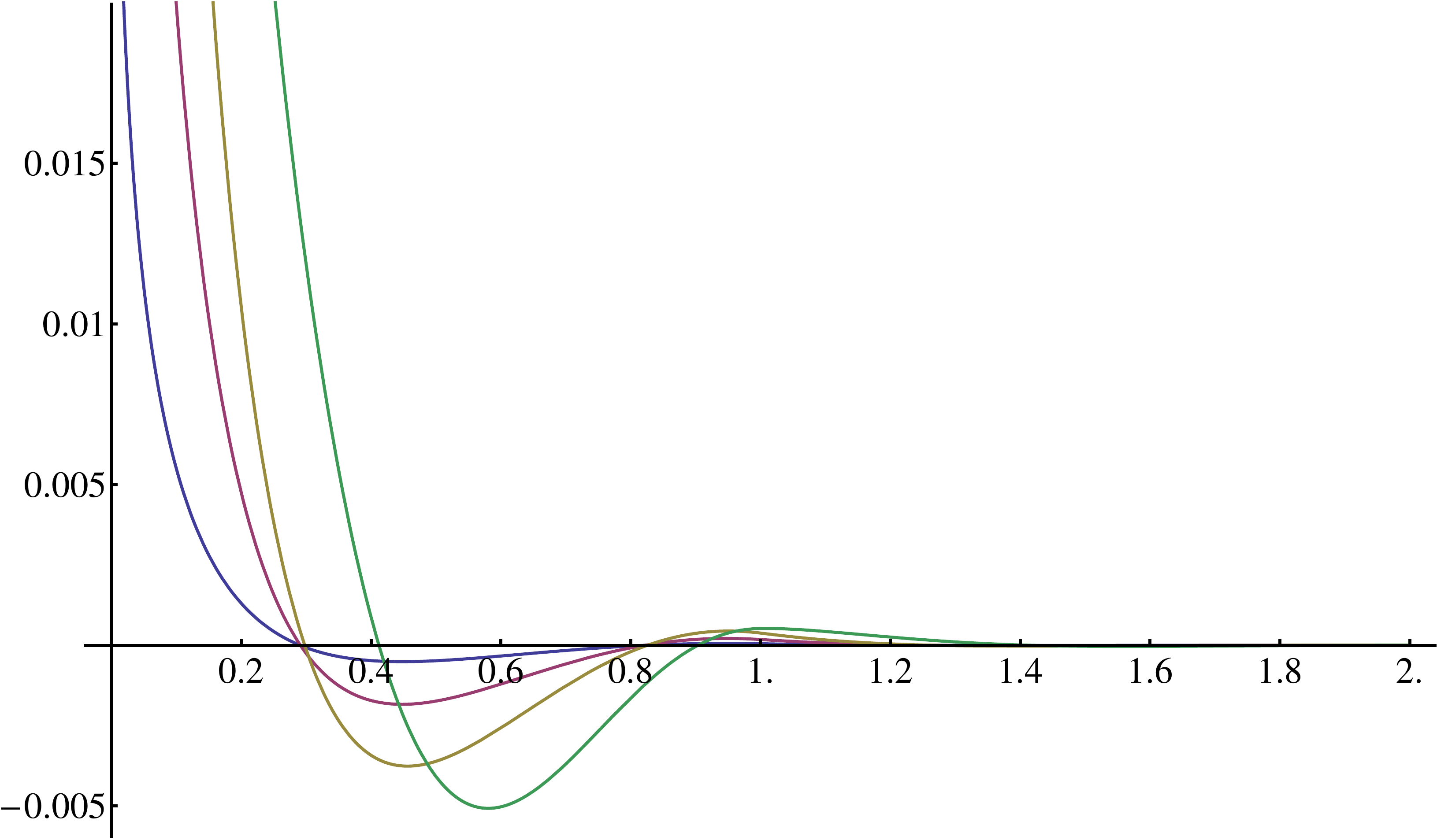}&
\includegraphics[width=0.49\textwidth]{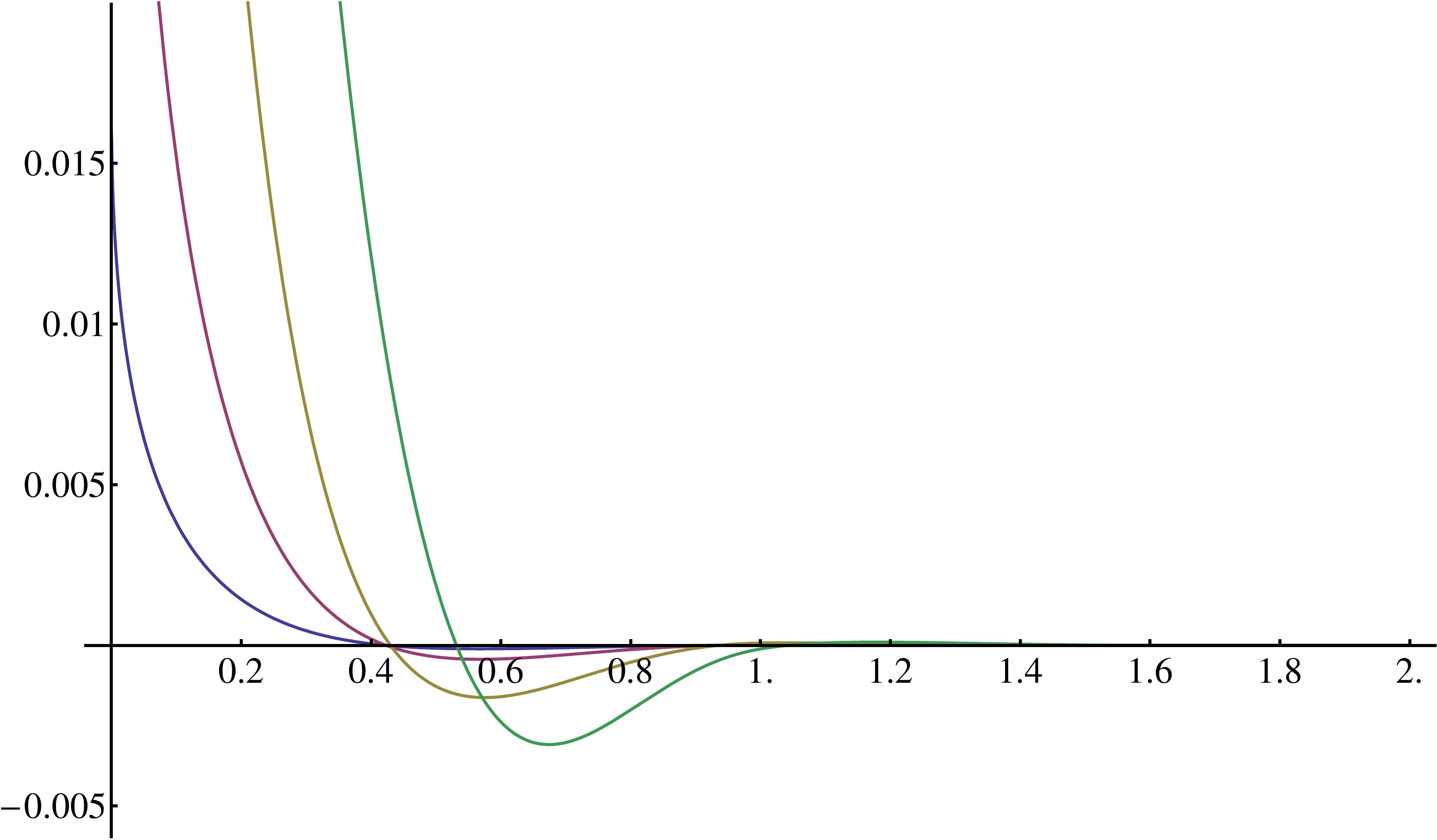}\\
\footnotesize \parbox{0.35\textwidth}{\centering $G_\lambda(x)$ for $\alpha = 1.1$ and $\lambda = 1/16$ (blue), $\lambda = 1/4$ (purple), $\lambda = 1$ (brown) and $\lambda = 4$ (green)}&
\footnotesize \parbox{0.35\textwidth}{\centering $G_\lambda(x)$ for $\alpha = 1.5$ and $\lambda = 1/16$ (blue), $\lambda = 1/4$ (purple), $\lambda = 1$ (brown) and $\lambda = 4$ (green)}\\[0.5em]
\includegraphics[width=0.49\textwidth]{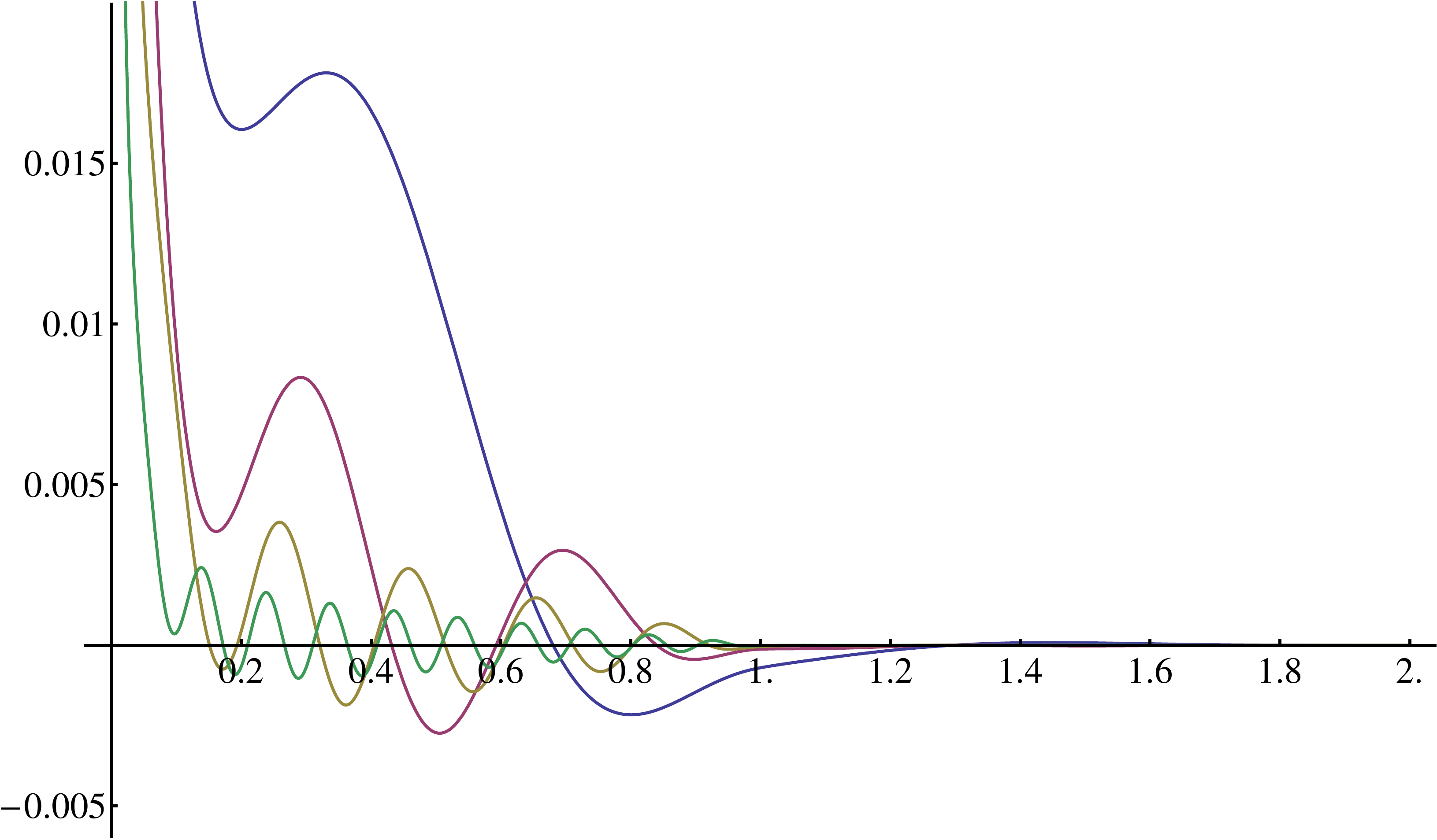}&
\includegraphics[width=0.49\textwidth]{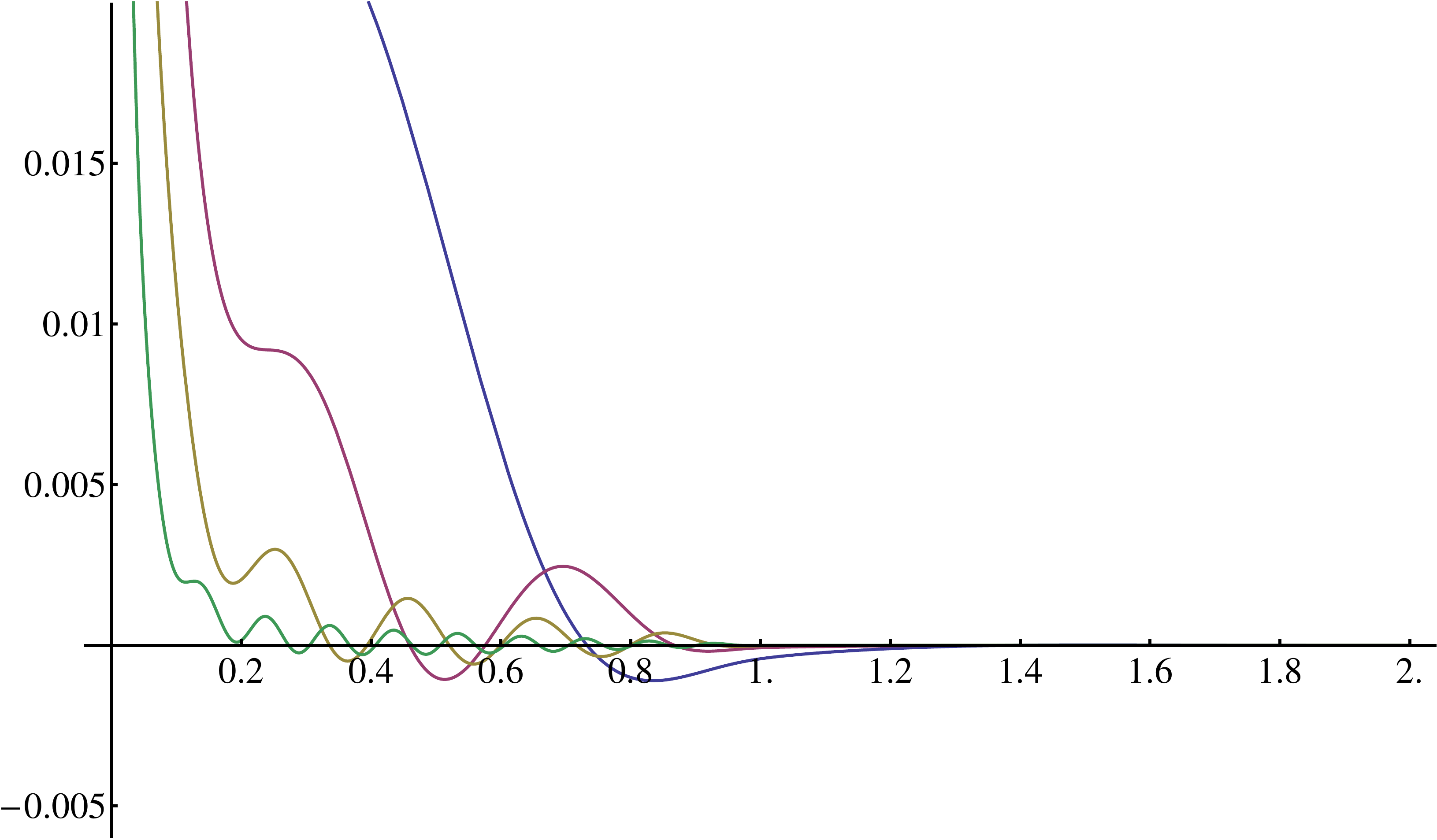}\\
\footnotesize \parbox{0.35\textwidth}{\centering $G_\lambda(x)$ for $\alpha = 1.1$ and $\lambda = 8$ (blue), $\lambda = 16$ (purple), $\lambda = 32$ (brown) and $\lambda = 64$ (green)}&
\footnotesize \parbox{0.35\textwidth}{\centering $G_\lambda(x)$ for $\alpha = 1.5$ and $\lambda = 8$ (blue), $\lambda = 16$ (purple), $\lambda = 32$ (brown) and $\lambda = 64$ (green)}
\end{tabular}
\caption{Plots for the truncated $\alpha$-stable process for $\alpha = 1.1$ and $\alpha = 1.5$. Note that $|G_\lambda(x)| \le G_\lambda(0)$, for all $\lambda > 0$, $x \in \R$, but $G_\lambda(x)$ fails to be everywhere nonnegative.}
\label{fig:3}
\end{figure}

\begin{example}
Consider the truncated symmetric $\alpha$-stable L\'evy process $X_t$, that is, a pure-jump L\'evy process $X_t$ with L\'evy measure $c |x|^{-1 - \alpha} \ind_{(-1, 1)}(x)$ for some $c > 0$. Then,
\formula{
 \psi(\xi) = \Psi(\sqrt{\xi}) & = c \int_{-1}^1 \frac{1 - \cos(\sqrt{\xi} x)}{|x|^{1 + \alpha}} \, dx = 2 c \xi^{\alpha/2} \int_0^{\sqrt{\xi}} \frac{1 - \cos s}{s^{1 + \alpha}} \, ds .
}
Clearly, $\psi'(\xi) > 0$ for all $\xi > 0$. We claim that $\psi''(\xi) \le 0$ for all $\xi > 0$. We have
\formula[eq:truncated]{
 \frac{2 \psi''(\xi)}{c} & = \frac{\sin \sqrt{\xi}}{\xi^{3/2}} - (2 - \alpha) \, \frac{1 - \cos \sqrt{\xi}}{\xi^2} - \frac{\alpha (2 - \alpha)}{\xi^{2 - \alpha/2}} \int_0^{\sqrt{\xi}} \frac{1 - \cos s}{s^{1 + \alpha}} \, ds .
}
Since $\sin s \le s - s^3/6 + s^5/120$ and $1 - \cos s \ge s^2/2 - s^4/24$ for $s > 0$, we have
\formula{
 \frac{2 \psi''(\xi)}{c} & \le \frac{1}{\xi} - \frac{1}{6} + \frac{\xi}{120} - \frac{2 - \alpha}{2 \xi} + \frac{2 - \alpha}{24} - \frac{\alpha}{2 \xi} + \frac{\alpha (2 - \alpha)}{24 (4 - \alpha)} \\
 & = -\frac{1}{3 (4 - \alpha)} + \frac{\xi}{120} \, ,
}
and hence $\psi''(\xi) \le 0$ for $\xi \le 40 / (4 - \alpha)$. When $\xi > 40 / (4 - \alpha)$, then in particular $\xi > \sqrt{12}$, and therefore, by~\eqref{eq:truncated},
\formula{
 \frac{2 \psi''(\xi)}{c} & \le \frac{1}{\xi^{3/2}} - \frac{\alpha (2 - \alpha)}{\xi^{2 - \alpha/2}} \int_0^{\sqrt{12}} \frac{s^2/2 - s^4/24}{s^{1 + \alpha}} \, ds \\
 & \le \frac{1}{\xi^{3/2}} \expr{1 - \frac{12^{1 - \alpha/2} \alpha}{(4 - \alpha)} \, \xi^{(\alpha - 1) / 2}} \le \frac{1}{\xi^{3/2}} \, (1 - \xi^{(\alpha - 1) / 2}) \le 0.
}
Our claim is proved.

In particular, we may apply Theorems~\ref{th:flambda}, \ref{th:spectral} and~\ref{th:tau} to truncated stable processes. Due to a rather involved description of $\Psi$, the computations are significantly harder in this case. See Figure~\ref{fig:3} for plots of $\thet_\lambda$ and $F_\lambda$.
\end{example}

\subsection*{Acknowledgements}
I would like to thank Kouji Yano for suggesting me the study of the spectral theory on $\cop$ and interesting discussions on this subject. I also thank Jacek Ma{\l}ecki for inventing and explaining me the proof of Theorem~1.9 in~\cite{bib:kmr11}, which made the present proof of Theorem~\ref{th:spectral} possible.

%
%                            ---------- o ----------
%

\end{document}